\RequirePackage[l2tabu, orthodox]{nag}
\documentclass[11pt]{article}%

\usepackage{amssymb}
\usepackage{amsfonts}
\usepackage{amsmath}
\usepackage{mathtools}
\usepackage{dsfont}

\usepackage{amsthm}
\usepackage{hyperref}
\usepackage{cleveref}

\usepackage{float}
\usepackage{graphicx}
\usepackage{caption}
\usepackage{epstopdf} 
\usepackage[all]{xy} 
\usepackage{tikz}

\usepackage[titletoc, title]{appendix} 
\usepackage{titling} 
\usepackage{url} 
\usepackage{color}
\usepackage{enumitem}

\usepackage[latin1]{inputenc}
\usepackage[a4paper]{geometry}
\usepackage{microtype}

\providecommand{\U}[1]{\protect\rule{.1in}{.1in}}
\newtheorem{theo}{Theorem}[section]
\newtheorem{prop}[theo]{Proposition}
\newtheorem{lem}[theo]{Lemma}
\newtheorem{cor}[theo]{Corollary}

\newtheorem{rem}[theo]{Remark}

\numberwithin{equation}{section}

\newcommand{\EE}{\mathbb{E}}

\newcommand{\PP}{\mathbb{P}}

\newcommand{\RR}{\mathbb{R}}

\newcommand{\Ba}{ {\mathcal B }}

\newcommand{\Na}{ {\mathcal N }}

\newcommand{\Ia}{ {\mathcal I }}

\newcommand{\Ta}{ {\mathcal T}}

\newcommand{\hide}[1]{}

\title{Universal tail profile of Gaussian multiplicative chaos}
\author{Mo Dick Wong\thanks{Statistical Laboratory, University of Cambridge.} \thanks{Fakult\"at f\"ur Mathematik, Universit\"at Wien.}}
\date{\today}
\begin{document}
\maketitle

\abstract{In this article we study the tail probability of the mass of Gaussian multiplicative chaos. With the novel use of a Tauberian argument and Goldie's implicit renewal theorem, we provide a unified approach to general log-correlated Gaussian fields in arbitrary dimension and derive precise first order asymptotics of the tail probability, resolving a conjecture of Rhodes and Vargas. The leading order is described by a universal constant that captures the generic property of Gaussian multiplicative chaos, and may be seen as the analogue of the Liouville unit volume reflection coefficients in higher dimensions.}

\tableofcontents

\section{Introduction}
Gaussian multiplicative chaos (GMC) was first constructed by Kahane \cite{Kah1985} in an attempt to provide a mathematical framework for the Kolmogorov-Obukhov-Mandelbrot model of energy dissipation in turbulence. The theory of (subcritical) GMC consists of defining and studying, for each $\gamma \in (0, \sqrt{2d})$, the random measure
\begin{align}\label{eq:GMC}
M_{\gamma}(dx) = e^{\gamma X(x) - \frac{\gamma^2}{2}\EE[X(x)^2]} dx,
\end{align}

\noindent where $X(\cdot)$ is a (centred) log-correlated Gaussian field on some domain $D \subset \RR^d$. The expression \eqref{eq:GMC} is formal because $X(\cdot)$ is not defined pointwise; instead it is only a random generalised function. It is now, however, well understood that $M_\gamma$ may be defined via a limiting procedure of the form
\begin{align*}
M_{\gamma}(dx) 
= \lim_{\epsilon \to 0} M_{\gamma, \epsilon}(dx)
= \lim_{\epsilon \to 0} e^{\gamma X_\epsilon(x) - \frac{\gamma^2}{2} \EE[X_\epsilon(x)^2]} dx
\end{align*}

\noindent where $X_{\epsilon}(\cdot)$ is some suitable sequence of smooth Gaussian fields that converges to $X(\cdot)$ as $\epsilon \to 0$. We refer the readers to e.g. \cite{Ber2017} for more details about the construction.

In recent years the theory of GMC has attracted a lot of attention in the mathematics and physics communities due to its wide array of applications -- it plays a central role in random planar geometry \cite{DMS2014, DS2011} and the mathematical formulation of Liouville conformal field theory (LCFT) \cite{DKRV2016}, appears as a universal limit in other areas such as random matrix theory \cite{Web2015, BWW2018, LOS2018, NSW2018}, and is even used as a model for Riemann zeta function in probabilistic number theory \cite{SW2016} or stochastic volatility in quantitative finance \cite{DRV2012}.

In spite of the importance of the theory, not much is known about the distributional properties of GMC. For instance, given a bounded open set $A \subset D$, one may ask what the exact distribution of $M_{\gamma}(A)$ is, but nothing is known except in very specific cases where specialised LCFT tools are applicable \cite{KRV2017, Rem2017, RZ2018}. Indeed even the regularity of the distribution (e.g. whether it has a density or not) is not known except for kernels with exact scale invariance \cite{RoV2010}.
\subsection{Main results}\label{subsec:main}
Define $M_{\gamma, g}(dx) = g(x) M_{\gamma}(dx)$ where  $g(x) \ge 0$ is continuous on $\overline{D}$. The goal of this paper is to derive the leading order asymptotics for
\begin{align} \label{eq:tail}
\PP \left(M_{\gamma, g}(A) > t\right)
\end{align}

\noindent for non-trivial\footnote{In the sense that $\int_A g(x) dx > 0$. In particular $A$ has non-trivial Lebesgue measure.} bounded open sets $A \subset D$ as $t \to \infty$. This may be seen as a first step towards the goal of understanding the full distribution of $M_{\gamma, g}(A)$, and will also highlight a new universality phenomenon of GMC. It is a standard fact in the literature that
\begin{align*}
\EE\left[M_{\gamma, g}(A)^p\right] < \infty \quad \Leftrightarrow \quad p < \frac{2d}{\gamma^2}
\end{align*}

\noindent and this suggests the possibility that the right tail \eqref{eq:tail} may satisfy a power law with exponent $2d/\gamma^2$. Our main result confirms this behaviour.

\begin{theo}\label{theo:main}
Let $\gamma \in (0, \sqrt{2d})$, $Q = \frac{\gamma}{2} + \frac{d}{\gamma}$ and $M_{\gamma, g}$ be the subcritical GMC associated with the Gaussian field $X(\cdot)$ with covariance
\begin{align} \label{eq:LGF_cov}
\EE[X(x) X(y)] = -\log |x-y| + f(x, y), \qquad \forall x, y \in D
\end{align}

\noindent where $f$ is a continuous function on $\overline{D} \times \overline{D}$. Suppose $f$ can be decomposed into
\begin{align}\label{eq:f_dec}
f(x, y) = f_+(x, y) - f_-(x, y)
\end{align}

\noindent where  $f_+, f_-$ are covariance kernels for some continuous Gaussian fields on $\overline{D}$. Then there exists some constant $\overline{C}_{\gamma, d} > 0$ independent of $f$ and $g$ such that for any bounded open set $A \subset D$,
\begin{align}\label{eq:main_subcritical}
\PP\left(M_{\gamma, g}(A) > t\right) 
\overset{t \to \infty}{=}
\left(\int_A e^{\frac{2d}{\gamma}(Q - \gamma) f(v, v)} g(v)^{\frac{2d}{\gamma^2}}dv\right)
\frac{\frac{2}{\gamma}(Q-\gamma)}{\frac{2}{\gamma}(Q-\gamma)+1}
\frac{\overline{C}_{\gamma, d}}{t^{\frac{2d}{\gamma^2}}}
+ o(t^{-\frac{2d}{\gamma^2}}).
\end{align}
\end{theo}

While the decomposition condition \eqref{eq:f_dec} may look intractable at first glance, it is implied by a more convenient criterion regarding higher regularity of $f$ (see \Cref{lem:holder_dec} or \cite{JSW2018} for more details about local Sobolev spaces $H_{loc}^s$). This is satisfied, for instance, by the Liouville quantum gravity measure in dimension $2$, i.e.
\begin{align*}
\mu_{\gamma}^{\mathrm{LQG}}(dx) = R(x; D)^{\frac{\gamma^2}{2}} M_{\gamma}(dx)
\end{align*}

\noindent where $M_{\gamma}(dx)$ is the GMC measure associated with the Gaussian free field with Dirichlet boundary conditions on $\partial D$, in which case $f(x, x) = R(x; D)$ is the conformal radius of $x$ in $D$. Such an application is not covered by any previously known results.

\begin{cor} \label{cor:f_reg}
Assume $f \in H_{\mathrm{loc}}^{s}(D \times D)$ for some $s > d$ instead of the decomposition condition \eqref{eq:f_dec} on $f$. Then the tail asymptotics \eqref{eq:main_subcritical} holds for any bounded open sets $A \subset D$ such that $\overline{A} \subset D$.
\end{cor}

\begin{proof}
Since we can always find another open set $A'$ such that $\overline{A} \subset A' \subset \overline{A'} \subset D$,  the decomposition condition on $f$, when restricted to $\overline{A}$, holds by \Cref{lem:holder_dec} and \Cref{theo:main} applies immediately.
\end{proof}

The constant $\overline{C}_{\gamma, d}$ that appears in the tail asymptotics \eqref{eq:main_subcritical} has various probabilistic representations which are summarised in \Cref{cor:prob_rep}, and we shall call it the reflection coefficient of Gaussian multiplicative chaos\footnote{evaluated at $\gamma$; see the general definition of $\overline{C}_{\gamma, d}(\alpha)$ in \Cref{app:prob_rep}.} as it may be seen as the $d$-dimensional analogue of the reflection coefficient in Liouville conformal field theory (LCFT), see \Cref{app:prob_rep}. Based on existing exact integrability results, we can even provide an explicit expression for $\overline{C}_{\gamma, d}$ when $d=1$ and $d=2$.

\begin{cor}[{cf. \cite[Section 4]{RV2017}}]
The constant $\overline{C}_{\gamma, d}$ in \eqref{eq:main_subcritical} is given by
\begin{align}\label{eq:coefficient}
\overline{C}_{\gamma, d}
= \begin{dcases}
\frac{(2\pi)^{\frac{2}{\gamma}(Q-\gamma)}}{\frac{\gamma}{2}(Q - \gamma) \Gamma\left(\frac{\gamma}{2}(Q - \gamma)\right)^{\frac{2}{\gamma^2}}}, & d = 1, \\
- \frac{\left(\pi \Gamma(\frac{\gamma^2}{4})/\Gamma(1-\frac{\gamma^2}{4})\right)^{\frac{2}{\gamma}(Q - \gamma)}}{\frac{2}{\gamma}(Q- \gamma)} \frac{\Gamma(-\frac{\gamma}{2}(Q-\gamma))}{\Gamma(\frac{\gamma}{2}(Q-\gamma))\Gamma(\frac{2}{\gamma}(Q-\gamma))}, & d= 2.
\end{dcases}
\end{align}
\end{cor}
\begin{proof}
The $d=2$ case follows from \cite{RV2017} which proves \eqref{eq:main_subcritical} when $f \equiv 0$ and $g \equiv 1$. By \Cref{theo:main}, our constant $\overline{C}_{\gamma, d}$ is independent of $f$ and therefore coincides with the Liouville unit volume reflection coefficient evaluated at $\gamma$, the value of which is given by the formula in \eqref{eq:coefficient}.

For $d=1$, this follows from \cite{Rem2017} which verifies the Fyodorov-Bouchard formula \cite{FB2008} that gives the exact distribution of the total mass of the GMC (associated with Gaussian free field with vanishing average over the unit circle) on the circle.
\end{proof}

\subsection{Previous work and our approach}
Despite being a very fundamental question, the tail probability of GMC has not been investigated very much in the literature. To our knowledge, the first result in this direction is established by Barral and Jin \cite{BJ2014} for the GMC associated with the exact scale invariant kernel $\EE[X(x) X(y)] = -\log|x-y|$ on the unit interval $[0, 1]$:
\begin{align*}
\PP(M_\gamma([0,1]) > t) = \frac{C_*}{t^{\frac{2}{\gamma^2}}} + o(t^{-\frac{2}{\gamma^2}})
\end{align*}

\noindent where the constant $C_* > 0$ is given by
\begin{align}\label{eq:Cstar}
C_* = \frac{2\gamma^2}{2 - \gamma^2} \frac{\EE \left[ M_{\gamma}([0,1])^{\frac{2}{\gamma^2}-1} M_\gamma([0, \frac{1}{2}]) - M_\gamma([0, \frac{1}{2}])^{\frac{2}{\gamma^2}}\right]}{\log 2}.
\end{align}

\noindent The issue about their approach is that they rely heavily on the exact scale invariance of the kernel and the symmetry of the unit interval in order to derive a stochastic fixed point equation, and it is not clear how their method may be generalised.

A recent paper \cite{RV2017} by Rhodes and Vargas, who consider the whole-plane Gaussian free field (GFF) restricted to the unit disc (i.e. $\EE[X(x) X(y)] = -\log|x-y|$ on $D = \{x \in \RR^2: |x| < 1\}$), offers a new perspective for the tail problem. Their starting point is the localisation trick
\begin{align*}
\PP\left(M_{\gamma, g}(A) > t\right) = \int_A \EE \left[ \frac{1_{\{M_{\gamma, g}(v, A) > t\}}}{M_{\gamma, g}(v, A)}\right]g(v)dv, 
\quad M_{\gamma, g}(v, A) := \int_{A} \frac{e^{\gamma^2 f(x, v)} M_{\gamma, g}(dx)}{|x-v|^{\gamma^2}}
\end{align*}

\noindent which effectively pins down the $\gamma$-thick points of $X(\cdot)$, allowing one to express the dependence of the leading tail coefficient on the test set $A$ in a very explicit way, and in the end they are able to obtain \eqref{eq:main_subcritical} when $f$ only consists of the positive definite part, i.e. $f_-(x, y) \equiv 0$, in dimension $d \le 2$.

Our strategy is inspired by the ideas from the aforementioned works, but we have made several important changes as well as new input. Let us defer the details of our proof to \Cref{sec:proof} and just highlight the difference between our approach and previous attempts here.

\begin{itemize}
\item[(i)] The use of Tauberian theorem: we translate the problem of the asymptotics for $\PP(M_{\gamma, g}(A)>t)$ as $t \to \infty$ to the equivalent problem of the asymptotics of 
\begin{align*}
\EE\left[e^{-\lambda / M_{\gamma, g}(A)}\right]
= \int_A\EE \left[\frac{1}{M_{\gamma, g}(v, A)} e^{-\lambda /M_{\gamma, g}(v, A)}\right] g(v) dv
\end{align*}

\noindent as $\lambda \to \infty$ (here the equality comes from a similar localisation trick). Unlike the approach in \cite{RV2017}, the expectation we deal with does not involve any indicator functions, which makes our analysis (such as the ``removal of non-singularity" step) much simpler.

\item[(ii)] Gaussian interpolation: thanks to the absence of any indicator functions in
\begin{align}\label{eq:non_con}
\EE \left[\frac{1}{M_{\gamma, g}(v, A)} e^{-\lambda /M_{\gamma, g}(v, A)}\right],
\end{align}

\noindent there is hope to reduce our problem to the case where the underlying kernel is exact (i.e. $\EE[X(x) X(y)] = -\log |x-y|$). Unlike many estimates such as moment bounds in GMC, the expectation \eqref{eq:non_con} we are studying here concerns a function $F: x \mapsto x^{-1} e^{-\lambda x}$ which is not convex or concave. The lack of a convenient convex/concave modification of $F$ without affecting the behaviour of the expectation as $\lambda \to \infty$ means that the popular convexity inequality \eqref{eq:Gcomp} is not applicable, and Kahane's full interpolation formula \eqref{eq:Kahane_int} plays an indispensable role in our analysis.

\item[(iii)] The analysis of the exact kernel: without the localisation trick, \cite{BJ2014} has to proceed by generalising Goldie's implicit renewal theorem to a form that is applicable to $M_{\gamma}([0,1])$, and they also need to show that the constant $C_*$ in \eqref{eq:Cstar} is finite, the proof of which is not trivial. In contrast, we only need the precise asymptotics for the tail probability
\begin{align}\label{eq:need_bound}
\PP\left(\int_{|x| \le r} |x|^{-\gamma^2}M_{\gamma}(dx) > t \right)
\end{align}

\noindent which follows readily from Goldie's original result and a simple coupling argument.
\end{itemize}

The novel elements in our proof not only allow us to bypass many tedious computations in existing approaches, but also extend the tail result \eqref{eq:main_subcritical} in three directions, namely
\begin{itemize}
\item general open test sets $A$: our result holds for any open subsets $A$ without further regularity assumption, unlike \cite{RV2017} which requires a $C^1$-boundary due to intricacies in dealing with the indicator function;

\item general kernels \eqref{eq:LGF_cov}: the continuity argument in \cite{RV2017} may treat the case where $f(x,y)$ is positive definite in $d=2$ but completely breaks down as soon as the negative definite part $f_-(x, y)$ is non-trivial, whereas we circumvent this issue entirely by an extrapolation principle;

\item arbitrary dimension $d$: our method does not make use of any special decomposition of the log-kernel $-\log |x-y|$, unlike \cite{BJ2014} (which requires the cone construction in $d=1$) or \cite{RV2017} (which relies on a radial/lateral decomposition of GFF in $d=2$), and this allows a unified approach to all dimensions.
\end{itemize}

\Cref{theo:main} shares the same spirit of the result in \cite{RV2017} in the sense that we have successfully separated the dependence on the test set $A$ and the functions $f, g$ from the rest of the tail coefficient, and the constant $\overline{C}_{\gamma, d}$ captures any remaining dependence on $d$ and $\gamma$ and generic feature of GMC. The fact that we are unable to provide an explicit formula for $\overline{C}_{\gamma, d}$ for $d \ge 3$ should not be seen as a drawback of our approach -- explicit expressions are known for $d=1$ and $d=2$ only because the constant has an LCFT interpretation, and their formulae are found (independently of the study of tail probability) by LCFT tools which do not seem to have natural generalisation to higher dimension at the moment.

\subsection{On the relevance of the kernel decomposition}
Based on the continuity assumption of $f$, it is always possible to decompose $f$ into the difference of two positive definite functions: indeed
\begin{align*}
T_f: h(\cdot) \mapsto \int_D f(\cdot, y) h(y) dy
\end{align*}

\noindent is a symmetric Hilbert-Schmidt operator that maps $L^2(D)$ to $L^2(D)$ and by the standard spectral theory of compact self-adjoint operators there exist $\lambda_n \in \RR$ and $\phi_n \in L^2(D)$ such that $(T_f\phi_n)(x) = \lambda_n \phi_n(x)$, $|\lambda_n| \xrightarrow{n \to \infty}{0}$ and
\begin{align*}
f(x, y) &= \sum_{n=1}^\infty \lambda_n \phi_n(x) \phi_n(y)\\
&= \underbrace{\left(\sum_{n=1}^\infty |\lambda_n| \phi_n(x) \phi_n(y)1_{\{\lambda_n > 0\}} \right)}_{=:f_+(x, y)} 
- \underbrace{\left(\sum_{n=1}^\infty |\lambda_n| \phi_n(x) \phi_n(y)1_{\{\lambda_n < 0\}} \right)}_{=:f_-(x, y)}
\end{align*}

\noindent in $L^2(D)$.Therefore, the relevant question is to determine the least regularity on $f_\pm$ for the power-law profile \eqref{eq:main_subcritical} to hold. Our decomposition condition \eqref{eq:f_dec} requires $f_\pm$ to be kernels of some continuous Gaussian fields. As it turns out, we only use this technical assumption to obtain the following estimate (see for instance \Cref{cor:tail_bound}(ii)): 
\begin{itemize}
\item There exists some $r > 0$ and $C > 0$ such that for all $v \in D$ and $s \in [0,1]$

\begin{align}\label{eq:tb_rel}
\PP \left( \int_{B(v, r) \cap D} \frac{M_{\gamma}^s(dx)}{|x-v|^{\gamma^2}} > t\right) \le \frac{C}{t^{\frac{2d}{\gamma^2}-1}} \qquad \forall t > 0
\end{align}

where $M_{\gamma}^s(dx) = e^{\gamma Z_s(x) - \frac{\gamma^2}{2} \EE\left[Z_s(x)^2\right]}dx$ is the Gaussian multiplicative chaos associated with the log-correlated field $Z_s$ with covariance $\EE[Z_s(x) Z_s(y)] = -\log |x-y| + s f(x, y)$.
\end{itemize}

\noindent Inspecting the proof in \Cref{sec:proof}, this is the only assumption (other than the continuity of $f$) we need in order to apply dominated convergence in several places (such as \eqref{eq:varphi_diff}) which ultimately yields the desired power law. In other words our decomposition condition \eqref{eq:f_dec} may be relaxed so long as \eqref{eq:tb_rel} is satisfied, e.g. we may assume instead that
\begin{itemize}
\item The Gaussian fields $G_{\pm}$ associated with the kernels $f_\pm$ satisfy
\begin{align}\label{eq:GP_bound}
\PP \left( \sup_{x \in D} |G_{\pm}(x)| < \infty \right) > 0
\end{align}

\noindent (see \Cref{subsec:GPfacts} for various implications).
\end{itemize}

\noindent All the proofs in \Cref{sec:proof} will go through without any modification to cover this slightly more general setting (which obviously includes the case where $G_\pm$ are continuous on $\overline{D}$). We choose not to phrase \Cref{theo:main} this way because \eqref{eq:GP_bound} is less tractable and not necessarily much more general. Indeed when $f_\pm(x, y) = f_\pm(x-y)$ are continuous translation-invariant kernels, a classical result by Belyaev \cite{Bel1960} states that $G_\pm$ are either continuous or unbounded on any non-empty open sets\footnote{The theorem of Belyaev actually concerns stationary kernels in $d=1$, but this implies the statement in higher dimension because we may view $G_\pm$, with $d-1$ coordinates fixed, as Gaussian fields in $1$ dimension.}, and so \eqref{eq:GP_bound} is equivalent to the original condition \eqref{eq:f_dec} in the stationary setting. We also think that the decomposition condition \eqref{eq:f_dec} is a very natural assumption because for any $s \ge 0$, $\epsilon > 0$ and symmetric function $f(\cdot, \cdot) \in H^s(\RR^{2d})$, one can always find some symmetric function $\widetilde{f}(\cdot, \cdot) \in C_c^\infty(\RR^{2d})$, say by truncating suitable basis expansion (see also \cite[Lemma 2.2]{JSW2018}), such that $||f - \widetilde{f}||_{H^s(\RR^{2d})} < \epsilon$ and that the operator $T_{\widetilde{f}}$ is of finite rank, i.e. the decomposition condition \eqref{eq:f_dec} is satisfied by a ``dense collection" of covariance kernels of the form \eqref{eq:LGF_cov}.

To understand the importance of continuity at the level of the fields $G_\pm$, let us consider the simpler situation where $f = f_+$. We have
\begin{align*}
\EE\left[X(x) X(y)\right] = -\log |x-y| + f(x, y) \approx - \log |x-y| + f(v,v)
\end{align*}

\noindent on a ball of small radius $r > 0$ centred around $v \in A$. This says that $X(\cdot)$ is the sum of an exact scale invariant field $Y$ (with covariance $\EE[Y(x) Y(y)] = K(x, y) = -\log |x-y|$) and an independent field $G_+$ which locally behaves like an independent random variable $N_v \sim \Na(0, f(v, v))$, and this leads to
\begin{align}
\notag & \PP\bigg( \underbrace{\int_A \frac{e^{\gamma^2f(x, v)}M_{\gamma, g}(dx) }{|x-v|^{\gamma^2}}}_{=:M_{\gamma, g} (v, A)} > t\bigg)\\
\label{eq:tail_loc}
& \qquad \approx \PP\bigg( \underbrace{\int_{|x-v|\le r} \frac{e^{\gamma^2f(x, v)}M_{\gamma, g}(dx) }{|x-v|^{\gamma^2}}}_{=:M_{\gamma, g}(v, r)} > t\bigg)
\sim e^{\frac{2d}{\gamma}(Q-\gamma) f(v,v)} g(v)^{\frac{2d}{\gamma^2}-1} \frac{\overline{C}_{\gamma, d}}{t^{\frac{2d}{\gamma^2}-1}}
\end{align}

\noindent (see \Cref{cor:tail_bound} and \Cref{rem:cor_important}). This allows us to interpret
\begin{align*}
\PP\left(M_{\gamma, g}(A) > t\right) 
\sim
\left(\int_A e^{\frac{2d}{\gamma}(Q - \gamma) f(v, v)} g(v)^{\frac{2d}{\gamma^2}}dv\right)
\frac{\frac{2}{\gamma}(Q-\gamma)}{\frac{2}{\gamma}(Q-\gamma)+1}
\frac{\overline{C}_{\gamma, d}}{t^{\frac{2d}{\gamma^2}}}
\end{align*}

\noindent in the following way: if $M_{\gamma, g}(A)$ is extremely large, then most of its mass comes from a small neighbourhood $B(v, r) \subset A$ of some $\gamma$-thick point $v \in A$ of $X(\cdot)$, and this point $v$ is more likely to come from regions of higher density with respect to $g$ and/or of higher values of $f $, i.e. where $G_+$ has higher variance near $v$.

When $G_+$ is not continuous, the localisation intuition is not valid anymore and our method breaks down because \eqref{eq:GP_bound} is possibly false by Belyaev's dichotomy mentioned earlier. It may happen that \eqref{eq:tb_rel} is still valid, in which case the power-law profile will still hold, but it is unclear how to proceed with a Gaussian field $G_+$ that is only guaranteed to have a separable and measurable version but nothing else. We conjecture that the power law \eqref{eq:main_subcritical} remains true without the generalised decomposition condition \eqref{eq:GP_bound} based on two heuristics:
\begin{itemize}
\item Despite the possibility that $G_\pm$ are unbounded in every non-empty open set, $G_\pm$ are still measurable and Lusin's theorem suggests some ``approximate" continuity of the fields which is much weaker than the usual notion of continuity but is perhaps sufficient for studying integrals.
\item The construction of the GMC measure involves the mollification of the underlying log-correlated field. When $G_\pm$ are convolved with a smooth mollifier $\theta \in C_c^\infty(\RR^d)$, the new covariance kernels are differentiable which implies that the resulting fields are actually continuous.
\end{itemize}

\subsection{Critical GMCs and extremal processes: heuristics}
Let us abuse the notation and denote by $M_{\sqrt{2d}}$ the critical GMC (via Seneta--Heyde renormalisation\footnote{Our definition differs from the usual one by the factor $\sqrt{\pi/2}$ for aesthetic purpose.})
\begin{align*}
M_{\sqrt{2d}}(dx) = \lim_{\epsilon \to 0^+} \sqrt{\frac{\pi}{2}}\left(\EE[X_\epsilon(x)^2]\right)^{\frac{1}{2}} e^{\sqrt{2d}X_{\epsilon}(x) - d\EE\left[X_{\epsilon}(x)^2\right]}dx
\end{align*}

\noindent and similarly $M_{\sqrt{2d}, g}(dx) = g(x)M_{\sqrt{2d}}(dx)$. While a similar criterion for the existence of moments \cite{DRSV2014b}
\begin{align*}
\EE \left[M_{\sqrt{2d}, g}(A)^p \right] < \infty \qquad \Leftrightarrow \qquad p < 1
\end{align*}

\noindent has been known for critical GMC associated with general fields, previous attempts to understand the tail probability $\PP(M_{\sqrt{2d}, g}(A) > t)$ are again restricted to exact kernels so that the derivation via stochastic fixed point equation may be applied \cite{BKNSW2015}. By combining thetechniques in this paper with additional ingredients including fusion estimates of GMC that have appeared in \cite{DKRV2017, BW2018} , it should be possible to prove that

\begin{align}\label{eq:main_critical}
\PP \left(M_{\sqrt{2d}, g}(A) > t\right)
\overset{t \to \infty}{=}\frac{\int_A g(v) dv}{t \sqrt{2d}} + o(t^{-1}).
\end{align}

\noindent The analysis of the critical tail probability is not an entirely straightforward adaptation of the proof here, however, and we decide to present the critical result in a separate article \cite{Working} in order not to overload the current paper. Nevertheless, let us provide a heuristic proof of \eqref{eq:main_critical} in the case $d=2$ based on \Cref{theo:main}. Recall that for $\gamma \in (0, 2)$ we have
\begin{align*}
\overline{C}_{\gamma, 2} = - \frac{\pi^{\frac{4}{\gamma^2}-1} \left(\Gamma(\frac{\gamma^2}{4})/\Gamma(1-\frac{\gamma^2}{4})\right)^{\frac{4}{\gamma^2}-1}}{\frac{4}{\gamma^2}-1} \frac{\Gamma(\frac{\gamma^2}{4}-1)}{\Gamma(1 - \frac{\gamma^2}{4})\Gamma(\frac{4}{\gamma^2}-1)}.
\end{align*}

\noindent Using the property\footnote{This was first proved in $d=2$, for GFF with Dirichlet boundary conditions in \cite{APS2018}, and subsequently extended in \cite{JSW2018} to log-correlated fields \eqref{eq:LGF_cov} with $f \in H_{\mathrm{loc}}^{d + \epsilon}$ in dimension $d = 2$.} that
\begin{align*}
\frac{M_{\gamma}(dx)}{2 - \gamma} \xrightarrow{\gamma \to 2^-} 2M_{2}(dx)
\end{align*}

\noindent and that $\Gamma(x) = x^{-1} \Gamma(1+x) \overset{x \to 0}{\sim} x^{-1}$, we should expect
\begin{align*}
\PP\left(M_{2, g}(A) > t\right)
& \overset{\gamma \to 2^-}{\approx} \PP\left(M_{\gamma, g}(A) > (2 - \gamma) 2t\right)\\
& \overset{\gamma \to 2^-}{\sim} \left(\frac{4}{\gamma^2} - 1\right) \left(\frac{1 - \frac{\gamma^2}{4}}{\frac{\gamma^2}{4}}\right)^{\frac{4}{\gamma^2} - 1} \frac{\int_A g(v) dv}{((2-\gamma) \cdot 2t)^{\frac{4}{\gamma^2}}}
\overset{\gamma \to 2^-}{\sim} \frac{\int_A g(v) dv}{2t}.
\end{align*}

Unfortunately it seems impossible to justify the interchanging of the limits $\gamma \to 2^-$ and $t \to \infty$ to turn the above argument into a rigorous proof, and this is actually not the approach adopted in the separate paper. On the other hand, the constant $\overline{C}_{\gamma, d}$ is not explicitly known in higher dimension $d \ge 3$ but the heuristic here suggests the existence of a non-trivial limit:
\begin{align*}
\lim_{\gamma \to \sqrt{2d}^-} (\sqrt{2d} - \gamma)^{\frac{2d}{\gamma^2}} \overline{C}_{\gamma, d}
= \lim_{\gamma \to \sqrt{2d}^-} (\sqrt{2d} - \gamma) \overline{C}_{\gamma, d}
 \in (0, \infty).
\end{align*}

\paragraph{Connection to discrete Gaussian free field}

The tail probability of critical chaos is not only interesting in its own right but is also closely related to the study of extrema of log-correlated Gaussian fields, which has been an active area of research in the last two decades. For instance, it is known that the extremal process of a discrete Gaussian free field (DGFF) in $d=2$ converges to a Poisson point process with intensity $e^{-2x} \otimes Z(dx)$ for some random measure $Z(dx)$ \cite{BL2014, BL2016, BL2018} which has long been conjectured to be some constant multiple of the critical LQG measure $\mu_2^{\mathrm{LQG}}$, i.e.
\begin{align}\label{eq:conformal_rad}
Z(dx) \propto \mu_2^{\mathrm{LQG}}(dx)
= R(x; D)^2 M_2(dx), \qquad x \in D
\end{align}

\noindent where $M_2(dx)$ is the critical GMC associated with Gaussian free field with Dirichlet boundary condition. The random measure $Z(dx)$ is characterised (up to a deterministic multiplicative factor) by a set of properties, among which the Laplace-type estimate
\begin{align} \label{eq:laplace}
\lim_{\lambda \to 0^+} \frac{\EE\left[ Z(A) e^{-\lambda Z(A)}\right]}{-\log \lambda} = c \int_A R(x; D)^2 dx,
\end{align}

\noindent (where $c>0$ is independent of $A$) has been left unverified by $\mu_2^{\mathrm{LQG}}$ for several years until very recently in the revision of \cite{BL2014}. Here we suggest an approach slightly different from that in \cite{BL2014}: it is sufficient to first establish the statement that
\begin{align} \label{eq:LQG_tail}
\PP \left(\mu_2^{\mathrm{LQG}}(A) > t\right) \overset{t \to \infty}{=} \frac{c \int_A R(x; D)^2dx}{t} + o(t^{-1})
\end{align}

\noindent from which we conclude that the Laplace-type estimate holds by straightforward computation. We would like to point out that \eqref{eq:LQG_tail} is a strictly stronger statement and cannot be deduced from the estimate \eqref{eq:laplace} without additional assumption\footnote{The claim that \eqref{eq:laplace} implies \eqref{eq:LQG_tail} was incorrectly made in \cite{BL2014}.}.

\subsection{Outline of the paper} 
The remainder of the article is organised as follows.

In \Cref{sec:prelim} we compile a list of results that will be used in the proof of \Cref{theo:main}. This includes a collection of facts regarding separable Gaussian processes, log-correlated Gaussian fields and GMCs, Karamata's Tauberian theorem and auxiliary asymptotics, and random recursive equations.

In \Cref{sec:proof} we present the proof of \Cref{theo:main} which is divided into two parts. After sketching the idea of the localisation trick, we first establish the tail asymptotics for GMCs associated with exact kernels. We then apply Kahane's interpolation and extend the result to general kernels \eqref{eq:LGF_cov}.

We conclude the article with \Cref{app:prob_rep} where we define the reflection coefficient $\overline{C}_{\gamma, d}(\alpha)$ of Gaussian multiplicative chaos and prove that it is equivalent to the Liouville reflection coefficients in $d=2$.

\paragraph{Acknowledgement}
The author would like to thank R\'emi Rhodes and Vincent Vargas for suggesting the problem, and Nathana\"el Berestycki for useful discussions. The author is supported by the Croucher Foundation Scholarship and EPSRC grant EP/L016516/1 for his PhD study at Cambridge Centre for Analysis.

\section{Preliminaries} \label{sec:prelim}

\subsection{Basic facts of Gaussian processes} \label{subsec:GPfacts}
We collect a few standard results regarding Gaussian processes in the following theorem.

\begin{theo}\label{theo:GPfacts}
Let $(G_t)_{t \in \Ta}$ be a separable centred Gaussian process such that
\begin{align*}
\PP\left(\sup_{t \in \Ta} |G_t| < \infty\right) > 0.
\end{align*}

\noindent Then the following statements are true.
\begin{itemize}
\item Zero-one law: $\PP\left(\sup_{t \in \Ta} |G_t| < \infty \right) = 1$.
\item Finite moments: $\EE \left[ \sup_{t \in \Ta} |G_t| \right] < \infty$ and $\sigma^2 = \sigma^2(G) = \sup_{t \in \Ta} \EE\left[G_t^2\right] < \infty$.
\item Concentration: there exists some $c > 0$ such that for any $t \ge 0$, 
\begin{align}\label{eq:Borel}
\PP\left( \left| \sup_{t \in \Ta} |G_t| - \EE \left[\sup_{t \in \Ta} |G_t|\right]\right| > t\right) \le 2 e^{-c\frac{u^2}{\sigma^2}}.
\end{align}
\end{itemize}
\end{theo}

The lemma below is an easy consequence of \Cref{theo:GPfacts}.
\begin{lem} \label{lem:ctsGP}
Let $G(\cdot)$ be a continuous Gaussian field on some compact domain $K \subset \RR^d$, then the following are true.
\begin{itemize}
\item[(i)] There exists some $c > 0$ such that
\begin{align}\label{eq:fast_tail}
\PP\left( \sup_{x \in K} |G(x)| > t \right) \le \frac{1}{c} e^{-c t^2}, \qquad \forall t \ge 0.
\end{align}
\item[(ii)]
Let $x \in \mathrm{int}(K)$. For any monotone functions $\Psi: \RR \to \RR$ with at most exponential growth at infinity,
\begin{align} \label{eq:squeeze}
\lim_{r \to 0^+} \EE \left[ \Psi\left(\sup_{y \in B(x, r)} G(y)\right)\right]
= \lim_{r \to 0^+} \EE \left[ \Psi\left(\inf_{y \in B(x, r)} G(y)\right)\right]
= \EE\left[\Psi\left(G(x)\right)\right]
\end{align}
\end{itemize}
\end{lem}

\begin{proof}
Since $G(\cdot)$ is continuous on $K$, it is separable and satisfies $\sup_{x \in K} |G(x)| < \infty$ almost surely. By \Cref{theo:GPfacts} we have $\EE\left[ \sup_{x \in K} |G(x)|\right] < \infty$ and $\sigma^2(G) < \infty$. The tail in (i) can thus be obtained from the concentration inequality \eqref{eq:Borel}.

For (ii), note that by monotonicity we can split $\Psi$ into positive and negative parts $\Psi = \Psi_+ - \Psi_-$, such that $\Psi_{\pm}$ are monotone functions with at most exponential growth at infinity. Since we can deal with $\Psi_+$ and $\Psi_-$ separately, we may as well assume without loss of generality that $\Psi$ is non-negative. Now take $r_0 > 0$ such that $B(x, r_0) \in K$, and consider the case where $\Psi$ is non-decreasing. By \eqref{eq:fast_tail} and the assumption on the growth of $\Psi$ at infinity, we have
\begin{align*}
\EE \left[ \Psi\left(\sup_{y \in B(x, r_0)} G(y)\right)\right] < \infty.
\end{align*}

\noindent But then for any $r \in (0, r_0)$,
\begin{align*}
0 \le \inf_{y \in B(x, r)} \Psi(G(y)) \le \sup_{y \in B(x, r)} \Psi(G(y)) \le \sup_{y \in B(x, r_0)} \Psi(G(y))
\end{align*}

\noindent and \eqref{eq:squeeze} follows from the continuity of $G$ and dominated convergence. The case where $\Psi$ is non-increasing is similar.
\end{proof}

\subsection{Decomposition of Gaussian fields}
We mention a result concerning the decomposition of symmetric functions from the very recent paper \cite{JSW2018}. Let $f(x, y)$ be a symmetric function on $D \times D$ for some domain $D \subset \RR^d$. We say $f$ is in  the local Sobolev space $H_{\mathrm{loc}}^s(D \times D)$ of index $s > 0$ if $\kappa f$ is in $H^s(D \times D)$ for any $\kappa \in C_c^\infty(D \times D)$, i.e.
\begin{align*}
\int_{\RR^d} (1+|\xi|^2)^s|\widehat{(\kappa f)}(\xi)|^2 d\xi < \infty
\end{align*}

\noindent where $\widehat{(\kappa f)}$ is the Fourier transform of $\kappa f$ (see more details in \cite[Section 2]{JSW2018}). Then

\begin{lem}[{cf. \cite[Lemma 3.2]{JSW2018}}]\label{lem:holder_dec}
If $f \in H_{\mathrm{loc}}^s(D \times D)$ for some $s > d$, then there exist two centred, H\"older-continuous Gaussian processes $G_\pm$ on $\RR^d$ such that
\begin{align}
\EE[G_+(x) G_+(y)] + \EE[G_-(x) G_-(y)] = f(x, y), \qquad \forall x, y \in D'
\end{align}

\noindent for any bounded open set $D'$ such that $\overline{D'} \subset D.$
\end{lem}

This decomposition result has various important implications, one of which is the positive-definiteness of the logarithmic kernel. The following result may be seen as a trivial special case of \cite[Theorem B]{JSW2018} and has been known since \cite{RV2010}.

\begin{lem}\label{lem:pd_exact}
For each $L \in \RR$, there exists $r_d(L) > 0$ such that the kernel
\begin{align} \label{eq:L_exact}
K_L(x, y) = -\log |x-y| + L
\end{align}

\noindent is positive definite on $B(0, r_d(L)) \subset \RR^d$. In particular, for any $R > 0$ there exists some $L > 0$ such that $K_L$ is positive definite on $B(0, R)$.
\end{lem}

For the sake of convenience, we shall from now on call \eqref{eq:L_exact} the $L$-exact kernel, and when $L = 0$ we simply call $K_0(\cdot, \cdot)$ the exact kernel and write $r_d = r_d(0)$. The exact kernel will play a pivotal role as the reference point from which we extrapolate our tail result to general kernels in the subcritical regime.

\subsection{Gaussian multiplicative chaos} \label{sec:GMC}
Given a log-correlated Gaussian field \ref{eq:LGF_cov}, there are various equivalent constructions of the GMC measure $M_{\gamma}$. In the subcritical case $\gamma \in (0, \sqrt{2d})$, one approach is the regularisation procedure, which is first suggested in \cite{RoV2010} and then generalised/simplified in \cite{Ber2017}. The idea is to pick any suitable mollifier $\theta(\cdot)$ and define
\begin{align}\label{eq:GMC_reg}
M_{\gamma, \epsilon}(dx) = e^{\gamma X_\epsilon(x) - \frac{\gamma^2}{2} \EE[X_\epsilon(x)^2]} dx
\end{align}

\noindent where $X_{\epsilon}(\cdot) = X \ast \theta_{\epsilon}(\cdot)$ is a continuous Gaussian field on $D$. Then
\begin{theo}
For $\gamma \in (0, \sqrt{2d})$,  the sequence of measures $M_{\gamma, \epsilon}$ converges in probability to some measure  $M_{\gamma}$ in the weak$^*$ topology as $\epsilon \to 0^+$. The limit $M_{\gamma}$ is independent of the choice of the mollification $\theta$.
\end{theo}

We collect a few standard results in the literature of GMC. The first is the celebrated interpolation principle by Kahane.
\begin{lem}[\cite{Kah1985}] \label{lem:Kahane}
Let $\rho$ be a Radon measure on $D$, $X(\cdot)$ and $Y(\cdot)$ be two continuous centred Gaussian fields, and $F: \RR_+ \to \RR$ be some smooth function with at most polynomial growth at infinity. For $t \in [0,1]$, define $Z_t(x) = \sqrt{t}X(x) + \sqrt{1-t}Y_t(x)$ and
\begin{align}
\varphi(t) := \EE \left[ F(W_t)\right], \qquad
W_t := \int_D e^{Z_t(x) - \frac{1}{2}\EE[Z_t(x)^2]} \rho(dx).
\end{align}

\noindent Then the derivative of $\varphi$ is given by
\begin{equation}\label{eq:Kahane_int}
\begin{split}
\varphi'(t) & = \frac{1}{2} \int_D \int_D \left(\EE[X(x) X(y)] - \EE[Y(x) Y(y)]\right) \\
& \qquad \qquad \times \EE \left[e^{Z_t(x) + Z_t(y) - \frac{1}{2}\EE[Z_t(x)^2] - \frac{1}{2}\EE[Z_t(y)^2]} F''(W_t) \right] \rho(dx) \rho(dy).
\end{split}
\end{equation}

\noindent In particular, if
\begin{align*}
\EE[X(x) X(y)] \le \EE[Y(x) Y(y)] \qquad \forall x, y \in D,
\end{align*}

\noindent then for any convex $F: \RR_+ \to \RR$
\begin{align}\label{eq:Gcomp}
\EE \left[F\left(\int_D e^{X(x) - \frac{1}{2} \EE[X(x)^2]}\rho(dx)\right)\right]
\le
\EE \left[F\left(\int_D e^{Y(x) - \frac{1}{2} \EE[Y(x)^2]}\rho(dx)\right)\right].
\end{align}

\noindent and the inequality is reversed if $F$ is concave instead.
\end{lem}

While \Cref{lem:Kahane} is stated for continuous fields, it may be extended to log-correlated fields if we first apply it to mollified fields $X_{\epsilon}$ and $Y_{\epsilon}$ and take the limit $\epsilon \to 0^+$. Such argument will work immediately for comparison principles \eqref{eq:Gcomp} and we shall make no further remarks on that. For the interpolation principle \eqref{eq:Kahane_int} we only need the following weaker statement which may be extended to log-correlated fields in the same way.

\begin{cor}\label{cor:interpolate}
Under the same assumptions and notations in \Cref{lem:Kahane}, if there exists some $C>0$ such that
\begin{align*}
\left|\EE[X(x) X(y)] - \EE[Y(x) Y(y)]\right| \le C \qquad \forall x, y \in D,
\end{align*}

\noindent then
\begin{align*}
|\varphi'(t)| \le \frac{C}{2} \EE \left[ (W_t)^2 |F''(W_t)|\right]
\end{align*}

\noindent and consequently
\begin{align*}
|\varphi(1) - \varphi(0)| \le \frac{C}{2} \int_0^1 \EE \left[ (W_t)^2 |F''(W_t)|\right] dt.
\end{align*}
\end{cor}

The next result is a generalised criterion for the existence of moments of GMC.
\begin{lem}\label{lem:Seiberg}
Let $\gamma \in (0, \sqrt{2d})$, $Q = \frac{\gamma}{2} + \frac{d}{\gamma}$, $\alpha \in [0, Q)$ and $B(0, r) \subset D$. Then
\begin{align}\label{eq:Seiberg}
\EE \left[ \left(\int_{|x| \le r} |x|^{-\gamma \alpha} M_{\gamma}(dx)\right)^s \right] < \infty
\end{align}

\noindent if $s < \frac{2d}{\gamma^2} \wedge \frac{2}{\gamma}(Q - \alpha)$. In particular
\begin{align*}
\EE \left[ \left(\int_{|x| \le r} M_{\gamma}(dx)\right)^s \right] &< \infty, \qquad \forall s < \frac{2d}{\gamma^2}, \\
\text{and} \qquad \EE \left[ \left(\int_{|x| \le r} |x|^{-\gamma^2} M_{\gamma}(dx)\right)^s \right] &< \infty, \qquad \forall s < \frac{2d}{\gamma^2} - 1.
\end{align*}

\end{lem}

\begin{rem}
The bound on \eqref{eq:Seiberg} is uniform among the class of fields \eqref{eq:LGF_cov} with $\sup_{x, y \in D} |f(x, y)| \le C$ for some $C > 0$ by Gaussian comparison (\Cref{lem:Kahane}).
\end{rem}

\subsection{Tauberian theorem and related auxiliary results}
Let us record the classical Tauberian theorem of Karamata.
\begin{theo}[{\cite[Theorem XIII.5.3]{Fel1971}}] \label{theo:tau}
Let $f(d\cdot)$ be a non-negative measure on $\RR_+$, $F(t):= \int_0^t f(ds)$ and suppose
\begin{align*}
\widetilde{F}(\lambda) := \int_0^\infty e^{-\lambda t} f(dt)
\end{align*}

\noindent exists for $\lambda > 0$. If $L$ is slowly varying at infinity and $\rho \in [0, \infty)$, then
\begin{align}\label{eq:tau}
\widetilde{F}(\lambda) \overset{\lambda \to \infty}{\sim} \lambda^{-\rho} L(\lambda^{-1})
\qquad \Leftrightarrow \qquad
F(\epsilon) \overset{\epsilon \to 0^+}{\sim} \frac{1}{\Gamma(1+\rho)} \epsilon^\rho L(\epsilon).
\end{align}

\noindent The above is also true when we consider the asymptotics $\lambda \to 0^+$ and $\epsilon \to \infty$ instead.
\end{theo}

Our use of \Cref{theo:tau} is summarised in the following corollary.
\begin{cor}\label{cor:tau}
Let $U$ be a non-negative random variable, $C>0$ and $p > 0$. Then
\begin{align}\label{eq:tau_rv}
\PP(U > t) \overset{t \to \infty} \sim \frac{C}{t^p}
\qquad \Leftrightarrow \qquad 
\EE \left[e^{-\lambda / U}\right] \overset{\lambda \to \infty}{\sim} \frac{C \Gamma(1+p)}{\lambda^p}.
\end{align}
\end{cor}

\begin{proof}
Let $V = U^{-1}$. In the notation of \Cref{theo:tau}, we choose $f(ds) = \PP(V \in ds)$, $L \equiv C \Gamma(1+p)$ and $\epsilon = t^{-1}$ such that $\widetilde{F}(\lambda) = \EE \left[e^{-\lambda / U}\right]$ and $\widetilde{F}(\epsilon) = \PP(U > t)$, and our claim is now immediate.
\end{proof}

To save ourselves from repeated calculations, we shall collect a few basic estimates below. The first one concerns the Laplace transform estimate of a random variable with power-law tail.

\begin{lem}\label{lem:lap_co}
If $U$ is a non-negative random variable such that
\begin{align*}
\PP(U > t) \overset{t \to \infty}{\sim} \frac{C}{t^q}
\end{align*}

\noindent for some $C > 0$ and $q > 0$, then for any $p > 0$
\begin{align}\label{eq:lap_tail}
\EE\left[U^{-p} e^{-\lambda / U}\right] \overset{\lambda \to \infty}{\sim} \frac{q}{p+q} \frac{C\Gamma(p+q+1)}{\lambda^{p+q}}. 
\end{align}

\noindent If $\PP(U > t) \le C t^{-q}$ for all $t > 0$ instead, then there exists some $C' > 0$ such that 
\begin{align} \label{eq:lap_bound}
\EE\left[U^{-p} e^{-\lambda / U}\right]  \le \frac{C'}{\lambda^{p+q}}, \qquad \forall \lambda > 0.
\end{align}
\end{lem}

\begin{proof}
For any $t_0 > 0$, it is not difficult to see that there exists $c_0 > 0$ such that
\begin{align*}
\EE\left[U^{-p} e^{-\lambda / U} 1_{\{U \le t_0\}} \right] = O(e^{-c_0 \lambda}).
\end{align*}

\noindent For any $\epsilon > 0$, choose $t_0 > 0$ such that for all $t > t_0$ we have
\begin{align*}
\frac{C(1-\epsilon)}{t^q} \le \PP(U > t) \le \frac{C(1+\epsilon)}{t^q}.
\end{align*}

\noindent Using Fubini, we have
\begin{align*}
\EE\left[U^{-p} e^{-\lambda / U} 1_{\{U \ge t_0\}} \right]
&= \frac{1}{t_0^p}e^{-\lambda / t_0} \PP(U > t_0)
+ \int_{t_0}^\infty e^{-\lambda /t} \left(-\frac{p}{t^{p+1}} + \frac{\lambda}{t^{p+2}}\right) \PP(U > t) dt\\
&\le O(e^{-\lambda / t_0}) + C\int_{t_0}^\infty e^{-\lambda / t} \left(-\frac{p(1-\epsilon)}{t^{p+q+1}} + \frac{\lambda(1+\epsilon)}{t^{p+q+2}}\right) dt.
\end{align*}

\noindent Note that for any $m > 0$ we have
\begin{align*}
\int_{t_0}^\infty \frac{e^{-\lambda /t}}{t^{m + 2}} dt
& = \lambda^{-(1+m)} \int_0^{\lambda/ t_0} s^m e^{-s} ds
\overset{\lambda \to \infty} = (1+o(1)) \Gamma(1+m) \lambda^{-(m+1)}
\end{align*}

\noindent and therefore
\begin{align*}
\EE\left[U^{-p} e^{-\lambda / U} \right]
& \le \frac{C}{\lambda^{p+q}} \left[ -p(1-\epsilon) \Gamma(p+q) + (1+\epsilon) \Gamma(p+q+1)\right] + o(\lambda^{-(p+q)})\\
& \le  \left(\frac{Cq}{p+q} + (p+1)\epsilon\right)\frac{\Gamma(p+q+1)}{\lambda^{p+q}} + o(\lambda^{-(p+q)}).
\end{align*}

\noindent Similarly we have
\begin{align*}
\EE\left[U^{-p} e^{-\lambda / U} \right]
& \ge \left(\frac{Cq}{p+q} - (p+1)\epsilon\right)\frac{\Gamma(p+q+1)}{\lambda^{p+q}} + o(\lambda^{-(p+q)}).
\end{align*}

\noindent This means that
\begin{align*}
&\left(\frac{Cq}{p+q} - (p+1)\epsilon\right)\Gamma(p+q+1)
\le \liminf_{\lambda \to \infty} \lambda^{q+1} \EE\left[U^{-p} e^{-\lambda / U} \right]\\
& \qquad \qquad \le \limsup_{\lambda \to \infty} \lambda^{q+1} \EE\left[U^{-p} e^{-\lambda / U} \right]
\le \left(\frac{Cq}{p+q} + (p+1)\epsilon\right)\Gamma(p+q+1).
\end{align*}

\noindent Since $\epsilon > 0$ is arbitrary, we let $\epsilon \to 0^+$ and obtain \eqref{eq:lap_tail}. The claim \eqref{eq:lap_bound} is similar.
\end{proof}

We collect another Laplace transform estimate, the proof of which is similar to that of Lemma \ref{lem:lap_co} and is omitted.
\begin{lem}\label{lem:fake_tau}
If $U$ is a non-negative random variable such that
\begin{align*}
\PP(U > t) \overset{t \to \infty}{\sim} \frac{C}{t^q}
\end{align*}

\noindent for some $C > 0$ and $q > 0$, then
\begin{align}\label{eq:fake_tau}
\lim_{\lambda \to 0^+} \frac{\EE\left[U^q e^{-\lambda U}\right]}{-\log \lambda} = Cq.
\end{align}

If $\PP(U > t) \le C t^{-q}$ for all $t$ sufficiently large instead, then \eqref{eq:fake_tau} may be replaced by the statement that the limit superior is upper bounded by $Cq$.
\end{lem}

We also need the following elementary result, the proof of which is again skipped.
\begin{lem} \label{lem:aux}
Let $U, V$ be two non-negative random variables. Suppose there exists some $C > 0$ and $q > 0$ such that
\begin{align*}
(i) & \qquad \PP(U > t) \overset{t \to \infty}{\sim} C t^{-q}, \\
(ii) & \qquad \PP(V > t) \overset{t \to \infty}{\sim} o(t^{-p}) \qquad  \forall p > 0.
\end{align*}

\noindent Then the tail behaviour of $UV$ is given by
\begin{align*}
(iii) & \qquad \PP(UV > t) \overset{t \to \infty}{\sim} C \EE[V^q] t^{-q}. \qquad 
\end{align*}

\end{lem}

\begin{rem}
The converse of \Cref{lem:aux} is false: in general if we are given only conditions (ii) and (iii), we can only show that there exists some $C' > 0$ such that
\begin{align*}
\PP(U > t) \le C' t^{-q}
\end{align*}

\noindent which follows immediately from $\PP(UV > t) \ge \PP(U > t/a ) \PP(V > a)$ for any $a > 0$ such that $\PP(V>a) \ne 0$.
\end{rem}

\subsection{Random recursive equation}
Here we collect Goldie's implicit renewal theorem \cite{Gol1991} from the literature of random distributional equations.

\begin{theo} \label{theo:Goldie}
Let $M$ and $R$ be two independent non-negative random variables. Suppose there exists some $q > 0$ such that
\begin{itemize}
\item [(i)] $\EE[M^q] = 1$.
\item [(ii)] $\EE[M^q \log M] < \infty$.
\item[(iii)] The conditional law of $\log M$ given $M \ne 0$ is non-arithmetic.
\item[(iv)] $\int_0^\infty |\PP(R > t) - \PP(MR > t)| t^{q-1} dt < \infty$.
\end{itemize}

\noindent Then $\EE[M^q \log M] \in (0, \infty)$ and as $t \to \infty$,
\begin{align*}
\PP(R > t) = \frac{C}{t^{q}} + o(t^{-q})
\end{align*}

\noindent where the constant $C > 0$ is given by
\begin{align}
C &= \frac{1}{\EE[M^q \log M]} \int_0^\infty \left(\PP(R > t) - \PP(MR > t)\right) t^{q-1} dt.
\end{align}
\end{theo}

\Cref{theo:Goldie} will be used alongside the following coupling lemma.
\begin{lem}\label{lem:coupling}
Let $U, V$ be two non-negative random variables and $q > 0$. Then
\begin{align}\label{eq:coupling}
\int_0^\infty \left| \PP(U > t) - \PP(V > t) \right| t^{q-1}dt \le \frac{1}{q} \EE \left| U^q - V^q\right|.
\end{align}

\noindent Moreover, for any coupling of $(U, V)$ such that $\EE|U^q - V^q| < \infty$,
\begin{align}\label{eq:coupling2}
\int_0^\infty \left[ \PP(U > t) - \PP(V > t) \right] t^{q-1}dt = \frac{1}{q} \EE \left[ U^q - V^q\right].
\end{align}
\end{lem}

\begin{proof}
Suppose  $U, V$ are bounded by some constant $M > 0$. The inequality \eqref{eq:coupling} is then a simple consequence of
\begin{align*}
& |\PP(U > t) - \PP(V> t)|\\
& \qquad = \left| \PP(U > t, V > t) + \PP(U > t, V \le t) - \PP(U> t, V > t) - \PP(U \le t, V > t)\right|\\
& \qquad = \left|\PP(U > t, V \le t) -\PP(U \le t, V > t)\right|\\
& \qquad \le \PP(U > t, V \le t) +\PP(U \le t, V > t)\\
& \qquad = \PP(\max(U, V) > t) - \PP(\min(U, V) > t)
\end{align*}

\noindent combined with the fact that
\begin{align*}
\EE\left|U^q - V^q\right|
& =\EE\left[ \max(U, V)^q - \min(U, V)^q \right]\\
& = q \int_0^\infty t^{q-1} \left[\PP(\max(U, V) > t) - \PP(\min(U, V) > t)\right] dt.
\end{align*}

\noindent The equality \eqref{eq:coupling2} is trivial because $\EE[U^q], \EE[V^q]$ are all finite.

For $U, V$ that are not necessarily bounded but $\EE|U^q - V^q| < \infty$ (otherwise \eqref{eq:coupling} is trivial), we introduce a cutoff $M > 0$ and write $U_M = \min(U, M), V_M = \min(V, M)$. Then the previous discussion implies that
\begin{align*}
&\int_0^M \left| \PP(U > t) - \PP(V > t) \right| t^{q-1}dt
= \int_0^\infty \left| \PP(U_M > t) - \PP(V_M > t) \right| t^{q-1}dt \\
& \qquad \le \frac{1}{q} \EE \left| \max(U_M, V_M)^q - \min(U_M, V_M)^q\right|\\
& \qquad \le \frac{1}{q} \EE \left|(U^q - V^q) 1_{\{\max(U, V) \le M \}}\right|
+ \frac{1}{q} \EE \left[\left(M^q - \min(U, V)^q\right)1_{\{\max(U, V) \ge M\}}\right]\\
& \qquad \xrightarrow{M \to \infty} \frac{1}{q}\EE \left|U^q - V^q\right|
\end{align*}

\noindent by dominated convergence since both 
\begin{align*}
\left|(U^q - V^q) 1_{\{\max(U, V) \le M \}}\right|,
\qquad \left(M^q - \min(U, V)^q\right)1_{\{\max(U, V) \ge M\}}
\end{align*}

\noindent are bounded by $|U^q - V^q|$. We send $M \to \infty$ on the LHS of the above inequality and obtain \eqref{eq:coupling} by monotone convergence. The equality \eqref{eq:coupling2} may be proved by a similar cutoff argument.
\end{proof}

\section{Proof of Theorem \ref{theo:main}} \label{sec:proof}
This section is devoted to the proof of the tail asymptotics of subcritical GMC measures. As advertised earlier, our proof of Theorem \ref{theo:main} consists of two parts.
\begin{itemize} 
\item[1.] Tail asymptotics of reference measure (\Cref{sec:ref_measure}): we derive the leading order term of $\PP\left(\int_{|x| \le r} |x|^{-\gamma^2} \overline{M}_{\gamma, g}(dx) > t\right)$ for the chaos measure $\overline{M}_{\gamma, g}$ associated with the exact kernel. This will serve as an important estimate for the extrapolation principle as well as applications of dominated convergence.
\item[2.] Extrapolation principle (\Cref{sec:interpolate}): we explain how the estimates for certain expectations involving $\overline{M}_{\gamma, g}$ may be extended to $M_{\gamma, g}$ by Gaussian interpolation.
\end{itemize}

\noindent To conclude our proof, we shall apply a Tauberian argument and translate our intermediate results back to the desired claim concerning the tail probability of $M_{\gamma, g}(A)$.

\medskip

Let us commence with the localisation trick.
\begin{lem} \label{lem:localisation}
Let $A \subset D$ be a non-empty open subset. Then for any $t > 0$ and $\lambda > 0$,
\begin{align}
\label{eq:localisation}
\EE \left[ e^{-\lambda / M_{\gamma, g}(A)}\right]
& = \int_A \EE \left[ \frac{1}{M_{\gamma, g}(v, A)} e^{-\lambda / M_{\gamma, g}(v, A)} \right]g(v)dv, \\
\label{eq:loc_old}
\PP \left(M_{\gamma, g}(A) > t\right)
& \le \int_A \EE \left[ \frac{1}{M_{\gamma, g}(v, A)} 1_{\{M_{\gamma, g}(v, A) \ge t\}} \right]g(v)dv
\end{align}

\noindent where
\begin{align*}
M_{\gamma, g}(v, A) := \int_A \frac{e^{\gamma^2 f(x, v)} M_{\gamma, g}(dx)}{|x-v|^{\gamma^2}}.
\end{align*}
\end{lem}

\begin{proof}
For each $\epsilon > 0$, let $X_{\epsilon}$ be the mollified field with covariance $\EE[X_{\epsilon}(x) X_{\epsilon}(y)] = -\log \left(|x-y| \vee \epsilon \right) + f_{\epsilon}(x, y)$ where $f_{\epsilon}(x, y) \xrightarrow{\epsilon \to 0^+} f(x,y)$ pointwise (cf. \cite[Lemma 3.4]{Ber2017}). If $M_{\gamma, \epsilon}(dx)$ is the GMC associated to $X_{\epsilon}$ and $M_{\gamma, g, \epsilon}(dx) = g(x)M_{\gamma, \epsilon}(dx)$, then
\begin{align}
\EE \left[ e^{-\lambda / M_{\gamma, g}(A)}\right]
\notag & = \lim_{\epsilon \to 0^+} \EE \left[ e^{-\lambda / M_{\gamma, g, \epsilon}(A)}\right] \\
\notag & = \lim_{\epsilon \to 0^+} \EE \left[ \frac{M_{\gamma, g, \epsilon}(A)}{M_{\gamma, g, \epsilon}(A)} e^{-\lambda / M_{\gamma, g, \epsilon}(A)}\right] \\
\label{eq:loc_moll1}
& = \lim_{\epsilon \to 0^+} \int_A \EE \left[\frac{e^{\gamma X_{\epsilon}(v) - \frac{\gamma^2}{2} \EE\left[X_{\epsilon}(v)^2\right]}}{M_{\gamma, g, \epsilon}(A)} e^{-\lambda / M_{\gamma, g, \epsilon}(A)}\right] g(v) dv.
\end{align}

\noindent One may interpret $e^{\gamma X_{\epsilon}(v) - \frac{\gamma^2}{2} \EE\left[X_{\epsilon}(v)^2\right]}$ as a Radon-Nikodym derivative, and by applying the Cameron-Martin theorem, we can remove this exponential by shifting the mean of $X_{\epsilon}(\cdot)$ by $\EE\left[X_{\epsilon}(\cdot) \gamma X_{\epsilon}(v)\right] = \gamma \left(- \log\left(|\cdot - v| \vee \epsilon \right) + f(\cdot, v)\right)$, i.e.
\begin{align}\label{eq:loc_moll2}
\EE \left[\frac{e^{\gamma X_{\epsilon}(v) - \frac{\gamma^2}{2} \EE\left[X_{\epsilon}(v)^2\right]}}{M_{\gamma, g, \epsilon}(A)} e^{-\lambda / M_{\gamma, g, \epsilon}(A)}\right]
& = \EE \left[ \frac{1}{M_{\gamma, g, \epsilon}(v, A)} e^{-\lambda / M_{\gamma, g, \epsilon}(v, A)} \right]
\end{align}

\noindent where
\begin{align*}
M_{\gamma, g, \epsilon}(v, A) 
& = \int_A e^{\gamma X_{\epsilon}(x) + \gamma \EE\left[X_{\epsilon}(x) X_{\epsilon}(v)\right] - \frac{\gamma^2}{2}\EE[X_{\epsilon}(x)^2]} g(x) dx
= \int_A \frac{e^{\gamma^2 f_{\epsilon}(x, v)} M_{\gamma, g, \epsilon}(dx)}{\left(|x-v|\vee \epsilon\right)^{\gamma^2}}.
\end{align*}

\noindent Since $M_{\gamma, g, \epsilon}(v, A)$ converges to $M_{\gamma, g}(v, A)$ as $\epsilon \to 0^+$, \eqref{eq:loc_moll2} converges to the integrand in \eqref{eq:localisation}, and we can interchange the limit and integral in \eqref{eq:loc_moll1} to obtain \eqref{eq:localisation} by bounded convergence. The proof of \eqref{eq:loc_old} is similar\footnote{Indeed \eqref{eq:loc_old} is an equality if the distribution $M_{\gamma, g}(v, A)$ is continuous, but this is only proved in the special case when the covariance kernel is exact. We are happy with the inequality here because we only need the estimate $\PP(M_{\gamma, g}(A) > t) \le t^{-1} \int_A \PP(M_{\gamma, g}(v, A) \ge t) dv$ later.} and is skipped here.

\end{proof}

\subsection{The reference measure $\overline{M}_{\gamma}$} \label{sec:ref_measure}
Let $\overline{M}_{\gamma}^L$ be the GMC associated with the log-correlated field $Y_L$ with covariance $\EE[Y_L(x) Y_L(y)] = K_L(x, y) = -\log|x-y| + L$, which by Lemma \ref{lem:pd_exact} is positive definite on $B(0, r_d(L))$. We shall suppress the dependence on $L$ when we are referring to the exact kernel, i.e. $L = 0$. The main estimate in this subsection is the asymptotics of the tail probability of $\overline{M}_{\gamma}(0, r) := \int_{|x| \le r} |x|^{-\gamma^2} \overline{M}_{\gamma}(dx)$.

\begin{lem}\label{lem:ref_GMC}
There exists some constant $\overline{C}_{\gamma, d} > 0$ such that for any $r \in (0, r_d]$,
\begin{align}
\PP\left(\overline{M}_{\gamma}(0, r) > t\right) = \frac{\overline{C}_{\gamma, d}}{t^{\frac{2d}{\gamma^2} - 1}} + o(t^{-\frac{2d}{\gamma^2} +1}), \qquad t \to \infty.
\end{align}
\end{lem}

\begin{proof}
Pick $c \in (0, 1)$. Using the fact that
\begin{align*}
\left(Y(c x)\right)_{|x| \le r}
\overset{d}{=} \left(Y(x) + N_c\right)_{|x| \le r}
\end{align*}

\noindent where $N_c \sim \Na(0, -\log c)$ is an independent random variable, we see that
\begin{align}
\overline{M}_{\gamma}(0, cr)
\notag & = \int_{|x| < |cr|} e^{\gamma Y(x) - \frac{\gamma^2}{2} \EE[Y(x)^2]}\frac{dx}{|x|^{\gamma^2}}\\
\notag & = c^d\int_{|u| < |r|} e^{\gamma Y(cu) - \frac{\gamma^2}{2} \EE[Y(cu)^2]}\frac{du}{|cu|^{\gamma^2}}\\
\notag & \overset{d}{=} c^{d - \gamma^2} e^{\gamma N_c - \frac{\gamma^2}{2} \EE[N_c^2]} \int_{|u| < |r|} e^{\gamma Y(u) - \frac{\gamma^2}{2} \EE[Y(u)^2]}\frac{du}{|u|^{\gamma^2}}\\
\label{eq:scaling}
& = c^{d - \frac{\gamma^2}{2}} e^{\gamma N_c} \overline{M}_{\gamma}(0, r).
\end{align}

For convenience, set $q = \frac{2d}{\gamma^2} - 1$ and write $M = c^{d - \frac{\gamma^2}{2}} e^{\gamma N_c}= c^{\frac{\gamma^2}{2} q}e^{\gamma N_c}$ and $R = \overline{M}_{\gamma}(0, r)$. We only need to show that conditions (i) -- (iv) in \Cref{theo:Goldie} are satisfied for the desired tail behaviour. Conditions (ii) and (iii) are trivial, while
\begin{align*}
\EE \left[ M^{q}\right]
& =c^{\frac{\gamma^2}{2} q^2} c^{-\frac{\gamma^2}{2}q^2}
= 1
\end{align*}

\noindent and so condition (i) is also satisfied. If we take $U = \overline{M}_{\gamma}(0, r)$, $V = \overline{M}_{\gamma}(0, cr)$, and 
\begin{align*}
W = U - V
= \int_{|x| \in [cr, r)} e^{\gamma Y(x) - \frac{\gamma^2}{2} \EE[Y(x)^2]} \frac{dx}{|x|^{\gamma^2}}
\le |cr|^{-\gamma^2} \overline{M}_{\gamma}(B(0, r)).
\end{align*}

\noindent then
\begin{align}
\int_0^{\infty} |\PP(R > t) - \PP(MR > t)| t^{q-1} dt 
\notag & = \int_0^{\infty} |\PP(U > t) - \PP(V > t)| t^{q-1} dt \\
\notag & \le \frac{1}{q} \EE \left| (V+W)^q - V^q \right|\\
\label{eq:tail_proof}
& \le 2^{q}\EE \left[ V^{q-1} W + W^{q}\right]
\end{align}

\noindent where the first inequality follows from \Cref{lem:coupling} and the second inequality from the elementary estimate
\begin{align*}
(V+W)^q - V^q 
& \le q \max\left(V^{q-1}, (V+W)^{q-1}\right) W
\le q 2^{q} \left(V^{q-1} W + W^q \right).
\end{align*}

\noindent Since $\EE[W^{q+1-\epsilon}] < \infty$ for any $\epsilon > 0$ (in particular that $\EE[W^q] < \infty$), we have
\begin{align*}
\EE[V^{q-1} W] \le \EE \left[ V^{\frac{(q-1)(q+1 - \epsilon)}{q - \epsilon}}\right]^{1 - \frac{1}{q+1-\epsilon}} \EE\left[W^{q+1 - \epsilon}\right]^{\frac{1}{q+1 - \epsilon}} < \infty
\end{align*}

\noindent for $\epsilon$ sufficiently small so that $(q-1)(q+1-\epsilon) / (q-\epsilon) < q$. Then \eqref{eq:tail_proof} is finite and condition (iv) is also satisfied, and by \Cref{theo:Goldie} we obtain
\begin{align*}
\PP(\overline{M}_{\gamma}(0, r) > t) = \frac{\overline{C}_{\gamma, d}}{t^{q}} + o(t^{-q}).
\end{align*}

\end{proof}

We summarise various probabilistic representations of $\overline{C}_{\gamma, d}$ in the following corollary.
\begin{cor}\label{cor:prob_rep}
The constant $\overline{C}_{\gamma, d}$ has the following equivalent representations.
\begin{align}
\overline{C}_{\gamma, d}
\notag 
& = \lim_{t \to \infty} t^{\frac{2d}{\gamma^2}-1} \PP\left(\overline{M}_{\gamma}(0, r) > t\right)\\
\label{eq:coeff2}
& = \lim_{\lambda \to 0^+} \frac{1}{\frac{2d}{\gamma^2}-1} \frac{\EE\left[ \overline{M}_{\gamma}(0,r)^{\frac{2d}{\gamma^2}-1} e^{-\lambda \overline{M}_{\gamma}(0,r)}\right]}{-\log \lambda}\\
\label{eq:coeff3}
& = \frac{1}{-\frac{2}{\gamma^2}\left(d - \frac{\gamma^2}{2}\right)^2 \log c} \EE\left[\overline{M}_{\gamma}(0, r)^{\frac{2d}{\gamma^2}-1} - \overline{M}_{\gamma}(0, cr)^{\frac{2d}{\gamma^2}-1}\right], \qquad \forall c \in (0, 1).
\end{align}
\end{cor}

\begin{proof}
The first representation is an immediate consequence of \Cref{lem:ref_GMC}, and the second representation follows from \Cref{lem:fake_tau}. For the third representation, we recall from \Cref{theo:Goldie} and \Cref{lem:coupling} that
\begin{align*}
& \lim_{t \to \infty} t^{q} \PP\left(\overline{M}_{\gamma}(0, r) > t\right)\\
& \qquad = \frac{1}{\EE\left[c^{\frac{\gamma^2}{2}q^2} e^{\gamma q N_c} \left(\frac{\gamma^2}{2}q \log c + \gamma N_c\right) \right]}\frac{1}{q} \EE\left[\overline{M}_{\gamma}(0, r)^{q} - \overline{M}_{\gamma}(0, cr)^{q}\right]
\end{align*}

\noindent where $q = \frac{2d}{\gamma^2} - 1$ and $c \in (0, 1)$. Then it is straightforward to check that
\begin{align*}
\EE\left[c^{\frac{\gamma^2}{2}q^2} e^{\gamma q N_c} \left(\frac{\gamma^2}{2}q \log c + \gamma N_c\right) \right]
& = \frac{\gamma^2}{2} q \log c + \gamma \EE\left[\gamma q N_c^2\right]
= - \frac{\gamma^2}{2}q \log c
\end{align*}

\noindent which gives \eqref{eq:coeff3}.
\end{proof}

\begin{rem} \label{rem:prob_rep}
The fact that \eqref{eq:coeff3} holds regardless of $c \in (0,1)$ is not surprising. Indeed when $c = 2^{-N}$, we have
\begin{align*}
\EE\left[\overline{M}_{\gamma}(0, r)^{\frac{2d}{\gamma^2}-1} - \overline{M}_{\gamma}(0, cr)^{\frac{2d}{\gamma^2}-1}\right]
= \sum_{n=1}^N \EE\left[\overline{M}_{\gamma}(0, 2^{-(n-1)} r)^{\frac{2d}{\gamma^2}-1} - \overline{M}_{\gamma}(0, 2^{-n}r)^{\frac{2d}{\gamma^2}-1}\right]
\end{align*}

\noindent and the summand on the RHS does not change with $n$ because of the scaling property \eqref{eq:scaling}. The scaling property also explains why \eqref{eq:coeff3} is independent of $r \in (0, r_d)$ (as long as the exact kernel remains positive definite on $B(0,r)$).

\end{rem}

\Cref{lem:ref_GMC} has several useful implications.
\begin{cor}\label{cor:tail_bound}
The following are true.
\begin{itemize}
\item[(i)] For any $L \in \RR$ and $r \in (0, r_d(L)]$, let $\overline{M}_{\gamma}^L(0, r) = \int_{|x| \le r} |x|^{-\gamma^2} e^{\gamma^2 L} \overline{M}_{\gamma}^L(dx)$. We have, as $t \to \infty$,
\begin{align}\label{eq:L_tail}
\PP\left(\overline{M}_{\gamma}^L(0, r) > t \right) = e^{\frac{2d}{\gamma}(Q-\gamma)L}\frac{\overline{C}_{\gamma, d} }{t^{\frac{2d}{\gamma^2}-1}} + o(t^{-\frac{2d}{\gamma^2} + 1}).
\end{align}
\item[(ii)] Let $X$ be the log-correlated field in \Cref{theo:main}, and $A \subset D$ be a fixed, non-trivial open set. Then there exists some $C > 0$ independent of $v \in A$ such that
\begin{align}\label{eq:tail_bound}
\PP\left( M_{\gamma, g}(v, A) > t \right) \le \frac{C}{t^{\frac{2d}{\gamma^2} - 1}} \qquad \forall t > 0.
\end{align}
\end{itemize}
\end{cor}

\begin{rem}\label{rem:cor_important}
The tail \eqref{eq:L_tail} suggests how $\PP\left(M_{\gamma, g}(v, A) > t\right)$ should behave asymptotically as $t \to \infty$. As we shall see in the proof, we can pick any $r > 0$ such that $B(v, r) \subset A$ and consider instead $\PP\left(M_{\gamma, g}(v, r) > t\right)$ without changing the asymptotic behaviour. When $r$ is small, the covariance structure of $X$ looks like $-\log |x-y| + f(v, v) = K_{f(v, v)}(x, y)$ locally in $B(v, r)$ and we should expect
\begin{align} \label{eq:tail_claim}
\PP\left( M_{\gamma, g}(v, r) > t \right) \sim e^{\frac{2d}{\gamma}(Q-\gamma) f(v, v)} g(v)^{\frac{2d}{\gamma^2} - 1}\frac{\overline{C}_{\gamma, d}}{t^{\frac{2d}{\gamma^2}-1}}.
\end{align}

\noindent It is not hard to verify this claim when $f$ is the covariance of some continuous Gaussian field, but the situation becomes trickier under the setting of \Cref{theo:main} where we only assume that $f = f_+ - f_-$ is the difference of two such covariance kernels. We shall therefore not attempt to prove \eqref{eq:tail_claim} here.

\end{rem}

\begin{proof}[Proof of \Cref{cor:tail_bound}]
For convenience, let $q = \frac{2d}{\gamma^2}-1 = \frac{2}{\gamma}(Q-\gamma)$.
\begin{itemize}
\item[(i)] For any $c, \theta \in (0, 1)$, we have
\begin{align*}
& \PP\left(\overline{M}_{\gamma}^L(0, cr) > t \right)
\le \PP\left(\overline{M}_{\gamma}^L(0, r) > t \right)\\
& \qquad \le \PP\left(\overline{M}_{\gamma}^L(0, cr) > (1-\theta) t \right)
+ \PP\left(\overline{M}_{\gamma}^L(0, B(0,r) \setminus B(0, cr)) > \theta t \right)
\end{align*}

\noindent where $\overline{M}_{\gamma}^L(0, A) := \int_{A} |x|^{-\gamma^2} e^{\gamma^2 L} \overline{M}_{\gamma}^L(dx)$. Since
\begin{align*}
\EE \left[\overline{M}_{\gamma}^L(0, B(0,r) \setminus B(0, cr))^p \right]
\le (cr)^{-p\gamma^2} \EE \left[\overline{M}_{\gamma}^L(B(0, r))^p \right] < \infty \qquad \forall p < \frac{2d}{\gamma^2},
\end{align*}

\noindent the tail probability of the random variable $\overline{M}_{\gamma}^L(0, B(0,r) \setminus B(0, cr))$ decays faster than $t^{-q}$ as $t \to \infty$ by Markov's inequality, and therefore
\begin{align*}
& \liminf_{t\to\infty} t^{q}\PP\left(\overline{M}_{\gamma}^L(0, cr) > t \right)
\le \liminf_{t \to \infty} t^{q}\PP\left(\overline{M}_{\gamma}^L(0, r) > t \right)\\
& \qquad \le \limsup_{t \to \infty} t^{q}\PP\left(\overline{M}_{\gamma}^L(0, r) > t \right)
\le \limsup_{t \to \infty} t^{q}\PP\left(\overline{M}_{\gamma}^L(0, cr) > (1-\theta) t \right).
\end{align*}

\noindent As $\theta \in (0, 1)$ is arbitrary, if $\PP\left(\overline{M}_{\gamma}(0, r) > t\right) \sim C t^{-q}$ for some $C > 0$, then $C$ must be independent of $r \in (0, r_d(L)]$. We may thus assume $r> 0$ to be as small as we like (but independent of $t$) without loss of generality.

If $L \ge 0$, we may interpret $K_L(x, y) = K_0(x, y) + L$ as the sum of the exact kernel and the variance of an independent random variable $\Na_{L} \sim \Na(0, L)$, and hence
\begin{align*}
\PP\left(\overline{M}_{\gamma}^L(0, r) > t\right)
= \PP\left(e^{\gamma \Na_L - \frac{\gamma^2}{2} L}\overline{M}_{\gamma}(0, r) > t\right)
\sim \frac{\overline{C}_{\gamma, d}\EE \left[\left(e^{\gamma \Na_L - \frac{\gamma^2}{2} L}\right)^{q}\right]}{t^q}
\end{align*}

\noindent by \Cref{lem:aux}, and $\EE \left[\left(e^{\gamma \Na_L - \frac{\gamma^2}{2} L}\right)^{q}\right] = e^{\frac{2d}{\gamma}(Q-\gamma)L}$.

If $L < 0$, we instead interpret $K_L(x, y) = - \log\left|e^{-L}(x-y)\right|$ as the exact kernel with coordinates scaled by $e^{-L}$. If we restrict ourselves to $x, y \in B(0, e^{-L} r_d)$ or equivalently $r \in (0, e^{-L}r_d]$, then
\begin{align*}
\PP\left(\overline{M}_{\gamma}^L(0, r) > t \right)
& = \PP\left(\int_{|x| \le r} |e^{-L}x|^{-\gamma^2} e^{\gamma Y(e^{-L}x) - \frac{\gamma^2}{2} \EE\left[Y(e^{-L}x )^2\right]} dx > t \right)\\
& = \PP \left( e^{dL}\overline{M}_{\gamma}(0, e^L r) > t\right)
\sim \frac{\overline{C}_{\gamma, d}e^{dqL}}{t^q}
\end{align*}

\noindent where $e^{dqL} = e^{\frac{2d}{\gamma}(Q-\gamma)L}$ as expected.
\item[(ii)] Let $r = r_d$. Then
\begin{align*}
\PP\left(M_{\gamma, g}(v, A) > t\right)
& \le \PP\left(M_{\gamma, g}(v, B(v,r) \cap D) > \frac{t}{2}\right) 
+ \PP\left(|r|^{-\gamma^2} e^{\gamma^2 L}M_{\gamma, g}(D) > \frac{t}{2}\right).
\end{align*}

\noindent Since $\EE\left[M_{\gamma, g}(D)^q\right] < \infty$ by \Cref{lem:Seiberg}, Markov's inequality implies that we only need to verify $\PP\left(M_{\gamma, g}(v, B(v,r) \cap D) > t\right) \le Ct^{-q}$ uniformly in $v$.

By (i), let $C > 0$ be such that
\begin{align*}
\PP\left( \overline{M}_{\gamma}(0, r) > t\right) \le \frac{C}{t^q} \qquad \forall t > 0.
\end{align*}

\noindent To go beyond exact kernels, we utilise the decomposition condition of $f$. Let $G_\pm(\cdot)$ be independent continuous Gaussian fields on $\overline{D}$ with covariance $f_{\pm}$, and introduce the random variables
\begin{align*}
R_+ = e^{\gamma \sup_{x \in D} G_+(x) + \gamma^2 \sup_{y, z \in D} |f(y, z)|},
\qquad R_- = e^{\gamma \inf_{x \in D} G_-(x) - \frac{\gamma^2}{2} \sup_{y \in D} |f_-(y, y)|}
\end{align*}

\noindent which possess moments of all orders by \Cref{lem:ctsGP}. Let $a > 0$ be such that
\begin{align*}
P_{R_-} := \PP(R_- > a) > 0. 
\end{align*}

\noindent Since $\EE[X(x)X(y)] + f_-(x, y) = K_0(x-v, y-v) + f_+(x, y)$, we have
\begin{align*}
& \PP(M_{\gamma, g}(v, B(v, r) \cap D) > t)
\le P_{R_-}^{-1} \PP( R_- M_{\gamma, g}(v, B(v, r) \cap D) > at)\\
& \quad \le P_{R_-}^{-1} \PP\left( \int_{B(v, r) \cap D} \frac{e^{\gamma^2 f(x, v)}e^{\gamma G_-(x) - \frac{\gamma^2}{2} \EE[G_-(x)^2]}}{|x-v|^{\gamma^2}} M_{\gamma}(dx)> \frac{at}{||g||_\infty}\right)\\
& \quad = P_{R_-}^{-1} \PP\left( \int_{B(0, r) \cap (D - v)} e^{\gamma^2 f(x+v, v)}e^{\gamma G_+(x+v) - \frac{\gamma^2}{2} \EE[G_+(x+v)^2]}\frac{\overline{M}_{\gamma}(dx)}{|x|^{\gamma^2}} > \frac{at}{||g||_\infty}\right)\\
& \quad \le P_{R_-}^{-1} \PP\left( R_+ \overline{M}_{\gamma}(0, r) > \frac{at}{||g||_\infty}\right)\\
& \quad \le P_{R_-}^{-1} \EE \left[\PP\left(\overline{M}_{\gamma}(0, r) > \frac{at}{||g||_\infty R_+} \Big| R_+\right)\right]
\le  P_{R_-}^{-1}\frac{ C (||g||_\infty/a)^q  \EE\left[R_+^q \right]}{t^q}.
\end{align*}

\noindent The coefficient $P_{R_-}^{-1} C (||g||_\infty/a)^{q}  \EE\left[R_+^q \right] < \infty$ is independent of $v$ so we are done.
\end{itemize}
\end{proof}

\subsection{The extrapolation principle} \label{sec:interpolate}
In this subsection we show the existence of the limit
\begin{align*}
\lim_{\lambda \to \infty} \lambda^{\frac{2d}{\gamma^2}} 
\EE\left[M_{\gamma, g}(v, A)^{-1} e^{-\lambda / M_{\gamma, g}(v, A)}\right]
\end{align*}

\noindent and establish a formula for it.

\paragraph{Step 1: removal of non-singularity.} We first show that
\begin{lem} \label{lem:rem_nonsing}
For any $r > 0$ such that $B(v, r) \subset A$, 
\begin{align} \label{eq:rem_nonsing}
\EE\left[M_{\gamma, g}(v, A)^{-1} e^{-\lambda / M_{\gamma, g}(v, A)}\right]
\overset{\lambda \to \infty}{=} \EE\left[M_{\gamma, g}(v, r)^{-1} e^{-\lambda / M_{\gamma, g}(v, r)}\right] + o(\lambda^{-\frac{2d}{\gamma^2}}).
\end{align}
\end{lem}

We emphasise that the error in \eqref{eq:rem_nonsing} need not be uniform in $v$ or $r$.

\begin{proof}
Starting with the localisation inequality \eqref{eq:loc_old}, we know by the uniform bound \eqref{eq:tail_bound} from \Cref{cor:tail_bound} that
\begin{align*}
\PP(M_{\gamma, g}(A) > t) \le \int_A \frac{1}{t} \PP\left(M_{\gamma, g}(v, A) \ge t\right)g(v)dv \le \frac{C\int_A g(v)dv}{t^{\frac{2d}{\gamma^2}}}
\end{align*}

\noindent for all $t > 0$. In particular 
\begin{align*}
\PP\left(M_{\gamma, g}(v, A \setminus B(v, r)) > t\right)
\le \PP\left(|r|^{-\gamma^2} M_{\gamma, g}(A) > t\right) \le \frac{C_{r,g}}{t^{\frac{2d}{\gamma^2}}} \qquad \forall t > 0
\end{align*}

\noindent for some $C_{r, g} > 0$. 

To finish our proof we only need to show matching upper/lower bounds for \eqref{eq:rem_nonsing}. For a lower bound, pick $\delta \in (0, 1)$ and
\begin{align*}
& \EE\left[M_{\gamma, g}(v, A)^{-1} e^{-\lambda / M_{\gamma, g}(v, A)}\right]
\ge \EE\left[\frac{e^{-\lambda / M_{\gamma, g}(v, r)}}{M_{\gamma, g}(v, r) + M_{\gamma, g}(v, A \setminus B(v, r))} \right]\\
& \qquad \ge \EE\left[\frac{e^{-\lambda / M_{\gamma, g}(v, r)}}{M_{\gamma, g}(v, r)}\left(1 + \frac{\lambda^{1-\delta}}{M_{\gamma, g}(v, r)}\right)^{-1}  1_{\{M_{\gamma, g}(v, r)\} \ge \lambda^{1 - \frac{\delta}{4}}, M_{\gamma, g}(v, A \setminus B(v, r)) \le \lambda^{1-\delta}\}}\right]\\
& \qquad \ge \left(1 - \lambda^{-\frac{3\delta}{4}}\right)\EE\left[\frac{e^{-\lambda / M_{\gamma, g}(v, r)}}{M_{\gamma, g}(v, r)} 1_{\{M_{\gamma, g}(v, r)\} \ge \lambda^{1 - \frac{\delta}{4}}, M_{\gamma, g}(v, A \setminus B(v, r)) \le \lambda^{1-\delta}\}}\right]\\
& \qquad = \left(1 - \lambda^{-\frac{3\delta}{4}}\right)\Bigg\{\EE\left[\frac{e^{-\lambda / M_{\gamma, g}(v, r)}}{M_{\gamma, g}(v, r)} \right]
-\EE\left[\frac{e^{-\lambda / M_{\gamma, g}(v, r)}}{M_{\gamma, g}(v, r)}
1_{\{M_{\gamma, g}(v, r)\le \lambda^{1 - \frac{\delta}{4}}\}}\right]\\
& \qquad \qquad \qquad \qquad \qquad  - 
\EE\left[\frac{e^{-\lambda / M_{\gamma, g}(v, r)}}{M_{\gamma, g}(v, r)}1_{\{M_{\gamma, g}(v, r) \ge \lambda^{1 - \frac{\delta}{4}}, M_{\gamma, g}(v, A \setminus B(v, r)) \ge \lambda^{1-\delta}\}}\right]\Bigg\}\\
\end{align*}

\noindent where
\begin{align*}
\EE\left[\frac{e^{-\lambda / M_{\gamma, g}(v, r)}}{M_{\gamma, g}(v, r)} 1_{\{M_{\gamma, g}(v, r)\le \lambda^{1 - \frac{\delta}{4}}\}}\right] 
\le \lambda^{-(1-\delta/4)} e^{-\lambda^{3\delta / 4}} 
&= o(\lambda^{-\frac{2d}{\gamma^2}})
\end{align*}

\noindent and
\begin{align*}
& \EE\left[\frac{e^{-\lambda / M_{\gamma, g}(v, r)}}{M_{\gamma, g}(v, r)}1_{\{M_{\gamma, g}(v, r) \ge \lambda^{1 - \frac{\delta}{4}}, M_{\gamma, g}(v, A \setminus B(v, r)) \ge \lambda^{1-\delta}\}}\right]\Bigg\}\\
&\qquad \le \lambda^{-(1-\delta/4)} \PP \left( M_{\gamma, g}(v, A \setminus B(v, r)) \ge \lambda^{1-\delta}\right)
\le C_r \lambda^{-(1-\delta)\left(\frac{2d}{\gamma^2}+1\right)}
\end{align*}

\noindent and so we just pick $\delta > 0$ small enough satisfying $(1-\delta)\left(\frac{2d}{\gamma^2} + 1\right) > \frac{2d}{\gamma^2}$ for our desired lower bound. 

As for the upper bound,
\begin{align*}
& \EE\left[M_{\gamma, g}(v, A)^{-1} e^{-\lambda / M_{\gamma, g}(v, A)}\right]
= \EE\left[M_{\gamma, g}(v, A)^{-1} e^{-\frac{\lambda}{M_{\gamma, g}(v, r)}\left(1 + \frac{M_{\gamma, g}(v, A \setminus B(v,r))}{M_{\gamma, g}(v, r)}\right)^{-1}}\right]\\
&\qquad \le \EE\left[M_{\gamma, g}(v, r)^{-1} e^{-\frac{\lambda}{M_{\gamma, g}(v, r)}\left(1 + \lambda^{-\frac{3\delta}{4}}\right)^{-1}} 1_{\{M_{\gamma, g}(v, r) \ge \lambda^{1-\frac{\delta}{4}}, M_{\gamma, g}(v, A \setminus B(0,r)) \le \lambda^{1 - \delta}\}} \right]\\
&\qquad \qquad + \underbrace{e^{-\frac{\lambda^{\delta/4}}{2}}\EE\left[M_{\gamma, g}(v, A)^{-1}\right]
+ \lambda^{-(1-\delta)}\PP\left(M_{\gamma, g}(v, A \setminus B(0,r)) > \lambda^{1 - \delta}\right)}_{ = o(\lambda^{-2d/\gamma^2})}
\end{align*}

\noindent where
\begin{align*}
&\EE\left[M_{\gamma, g}(v, r)^{-1} e^{-\frac{\lambda}{M_{\gamma, g}(v, r)}\left(1 + \lambda^{-\frac{3\delta}{4}}\right)^{-1}} 1_{\{M_{\gamma, g}(v, r) \ge \lambda^{1-\frac{\delta}{4}}, M_{\gamma, g}(v, A \setminus B(0,r)) \le \lambda^{1 - \delta}\}} \right]\\
& \qquad \le \underbrace{e^{\lambda^{-\delta/2}}}_{= 1+o(1)}\EE\Bigg[M_{\gamma, g}(v, r)^{-1} e^{-\frac{\lambda}{M_{\gamma}(v, r)}} 1_{\{M_{\gamma, g}(v, r) \ge \lambda^{1-\frac{\delta}{4}}, M_{\gamma, g}(v, A \setminus B(0,r)) \le \lambda^{1 - \delta}\}} \Bigg]\\
& \qquad \le (1+o(1))\EE\Bigg[M_{\gamma, g}(v, r)^{-1} e^{-\frac{\lambda}{M_{\gamma, g}(v, r)}}\Bigg] + o(\lambda^{-\frac{2d}{\gamma^2}})
\end{align*}

\noindent where the last inequality follows from similar calculations in the proof of the lower bound. This concludes the proof of \eqref{eq:rem_nonsing}.
\end{proof}

\paragraph{Step 2: extrapolation.}
For $s \in [0,1]$, define $Z_s(x) = \sqrt{s}X(x) + \sqrt{1-s} Y_{f(v, v)}(x - v)$, $M_{\gamma}^s(dx) = e^{\gamma Z_s(x) - \frac{\gamma^2}{2} \EE[Z_s(x)^2]}dx$ and
\begin{align}\label{eq:phi_def}
M_{\gamma, g}^s(v, r) := \int_{B(v, r)}\frac{e^{\gamma^2 f(v,v)} g(v) M_{\gamma}^s(dx)}{|x-v|^{\gamma^2}},
\qquad \varphi(s) := \EE \left[ \frac{1}{M_{\gamma, g}^s(v, r)} e^{-\lambda / M_{\gamma, g}^s(v, r)}\right]
\end{align}

\noindent where $r \in (0, r_d(f(v,v))]$. Our goal is to prove the following extrapolation result.
\begin{lem}\label{lem:lap_estimate}
Suppose $v \in D$ satisfies $g(v) > 0$. Then
\begin{align}\label{eq:lap_interpolate}
\lim_{\lambda \to \infty} \lambda^{\frac{2d}{\gamma^2}} \varphi(1)
& = \lim_{\lambda \to \infty} \lambda^{\frac{2d}{\gamma^2}} \varphi(0)\\
\notag & = \Gamma \left(1+\frac{2d}{\gamma^2}\right) e^{\frac{2d}{\gamma}(Q-\gamma) f(v, v)} g(v)^{\frac{2d}{\gamma^2}-1} \frac{\frac{2}{\gamma}(Q-\gamma)}{\frac{2}{\gamma}(Q-\gamma) + 1} \overline{C}_{\gamma, d}.
\end{align}

\noindent Furthermore,
\begin{align} 
\notag & \lim_{\lambda \to \infty} \lambda^{\frac{2d}{\gamma^2}} \EE \left[M_{\gamma, g}(v, r)^{-1} e^{-\lambda / M_{\gamma, g}(v, r)}\right]\\
\label{eq:lap_estimate} 
& \qquad = \Gamma \left(1+\frac{2d}{\gamma^2}\right) e^{\frac{2d}{\gamma}(Q-\gamma) f(v, v)} g(v)^{\frac{2d}{\gamma^2}-1}\frac{\frac{2}{\gamma}(Q-\gamma)}{\frac{2}{\gamma}(Q-\gamma) + 1} \overline{C}_{\gamma, d}.
\end{align}
\end{lem}

\begin{proof}
We first recall that the definition of $\varphi(t)$ depends on $r$ but the limits \eqref{eq:lap_interpolate}, if exist, do not because of \Cref{lem:rem_nonsing}. Also
\begin{align*}
\lim_{\lambda \to \infty} \lambda^{\frac{2d}{\gamma^2}} \varphi(0)
& = \lim_{\lambda \to \infty} \lambda^{\frac{2d}{\gamma^2}} \EE \left[\left(g(v)\overline{M}_{\gamma}^{f(v,v)}(v, r)\right)^{-1} e^{-\lambda /\left(g(v)\overline{M}_{\gamma}^{f(v,v)}(v, r)\right)}\right]\\
& = \Gamma \left(1+\frac{2d}{\gamma^2}\right) e^{\frac{2d}{\gamma}(Q-\gamma) f(v, v)} g(v)^{\frac{2d}{\gamma^2}-1} \frac{\frac{2}{\gamma}(Q-\gamma)}{\frac{2}{\gamma}(Q-\gamma) + 1} \overline{C}_{\gamma, d}
\end{align*}

\noindent by combining \Cref{cor:tail_bound} ($L = f(v, v)$) with \Cref{lem:lap_co}. From now on we shall  focus on the equality of the two limits \eqref{eq:lap_interpolate}.

For any $\epsilon > 0$ there exists some $r = r(\epsilon) \in (0, r_d(f(v, v))]$ such that
\begin{align}\label{eq:continuity}
\left| f(x, y) - f(v, v)\right| \le \epsilon 
\end{align}

\noindent for all $x, y \in B(v, r)$ by continuity. If we write $F(x) = x^{-1} e^{-\lambda / x}$, then $F''(x) = e^{-\lambda / x}\left(\frac{2}{x^{3}} - \frac{4\lambda}{x^4} + \frac{\lambda^2}{x^5}\right)$, and \Cref{cor:interpolate} yields
\begin{align}\label{eq:varphi_diff}
|\varphi(1) - \varphi(0)| 
\le \frac{\epsilon}{2} \int_0^1 \EE\left[ e^{-\lambda / M_{\gamma, g}^s(v, r)} \left(\frac{2}{M_{\gamma, g}^s(v, r)} + \frac{4\lambda}{M_{\gamma, g}^s(v, r)^2} + \frac{\lambda^2}{M_{\gamma, g}^s(v, r)^3}\right)\right]ds.
\end{align}

Going through the proof of \Cref{cor:tail_bound}(ii) again, we check that the same argument also suggests that there exists some $C > 0$ independent of $s \in [0,1]$ and $v \in D$ such that
\begin{align*}
\PP\left(M_{\gamma, g}^s(v, r) > t\right) \le \frac{C}{t^{\frac{2d}{\gamma^2}- 1}} \qquad \forall t > 0.
\end{align*}

\noindent By \Cref{lem:lap_co}, the integrand in \eqref{eq:varphi_diff} is uniformly bounded by $C' \lambda^{-\frac{2d}{\gamma^2}}$ for some $C' > 0$ which means that
\begin{align*}
\limsup_{\lambda \to \infty} \lambda^{\frac{2d}{\gamma^2}} |\varphi(1) - \varphi(0)|
\le \frac{C' \epsilon}{2}.
\end{align*}

\noindent Since $\epsilon > 0$ is arbitrary, we have $\lim_{\lambda \to \infty} \lambda^{\frac{2d}{\gamma^2}} \varphi(1) = \lim_{\lambda \to \infty} \lambda^{\frac{2d}{\gamma^2}} \varphi(0)$.

Finally, let $\epsilon, r > 0$ be chosen according to \eqref{eq:continuity} and the additional constraint that
\begin{align*}
\left|\frac{g(x)}{g(v)} - 1\right| \le \epsilon \qquad \forall x \in B(v, r)
\end{align*}

\noindent which is possible because $g(v) > 0$ and $g$ is continuous. Then
\begin{align*}
& \liminf_{\lambda \to \infty} \lambda^{\frac{2d}{\gamma^2}} \EE \left[M_{\gamma, g}(v, r)^{-1} e^{-\lambda / M_{\gamma, g}(v, r)}\right]\\
& \qquad  \ge \lim_{\lambda \to \infty} \lambda^{\frac{2d}{\gamma^2}} (1+\epsilon)^{-1} e^{-\gamma^2 \epsilon}\EE \left[M_{\gamma, g}^1(v, r)^{-1} e^{-\lambda (1+\epsilon) e^{\gamma^2 \epsilon} / M_{\gamma, g}^1(v, r)}\right]\\
& \qquad = \left((1+\epsilon)e^{\gamma^2 \epsilon}\right)^{-\left(1 + \frac{2d}{\gamma^2}\right)} \lim_{\lambda \to \infty} \lambda^{\frac{2d}{\gamma^2}} \varphi(1), \\
& \limsup_{\lambda \to \infty} \lambda^{\frac{2d}{\gamma^2}} \EE \left[M_{\gamma, g}(v, r)^{-1} e^{-\lambda / M_{\gamma, g}(v, r)}\right]\\
& \qquad \le \lim_{\lambda \to \infty} \lambda^{\frac{2d}{\gamma^2}} (1+\epsilon) e^{\gamma^2 \epsilon}\EE \left[M_{\gamma, g}^1(v, r)^{-1} e^{-\lambda (1+\epsilon)^{-1} e^{-\gamma^2 \epsilon} / M_{\gamma, g}^1(v, r)}\right]\\
& \qquad = \left((1+\epsilon)e^{\gamma^2 \epsilon}\right)^{\left(1 + \frac{2d}{\gamma^2}\right)} \lim_{\lambda \to \infty} \lambda^{\frac{2d}{\gamma^2}} \varphi(1).
\end{align*}

\noindent Given that the $\liminf$/$\limsup$ do not depend on $r$ by \Cref{lem:rem_nonsing}, $\epsilon$ can be made arbitrarily small and the claim \eqref{eq:lap_estimate} follows.
\end{proof}

\begin{proof}[Proof of \Cref{theo:main}]
By \Cref{cor:tail_bound} (ii) and \Cref{lem:lap_co}, we see that
\begin{align*}
\lambda^{\frac{2d}{\gamma^2}} \EE \left[M_{\gamma, g}(v, A)^{-1} e^{-\lambda / M_{\gamma, g}(v, A)}\right]
\le C' \qquad \forall v \in A.
\end{align*}

\noindent With an application of dominated convergence, the localisation identity \eqref{eq:localisation} yields
\begin{align*} 
& \lim_{\lambda \to \infty} \lambda^{\frac{2d}{\gamma^2}} \EE \left[ e^{-\lambda / M_{\gamma, g}(A)}\right]\\
& \qquad = \int_A \left( \lim_{\lambda\to \infty} \lambda^{\frac{2d}{\gamma^2}}\EE \left[M_{\gamma, g}(v, A)^{-1} e^{-\lambda / M_{\gamma, g}(v, A)} \right] \right) g(v)dv.
\end{align*}

\noindent We substitute the pointwise limit of the integrand from \Cref{lem:lap_estimate} and obtain 
\begin{align*}
& \lim_{\lambda \to \infty} \lambda^{\frac{2d}{\gamma^2}} \EE \left[ e^{-\lambda / M_{\gamma, g}(A)}\right] \\
& \qquad = \Gamma\left(1+ \frac{2d}{\gamma^2}\right) \left(\int_A e^{\frac{2d}{\gamma}(Q-\gamma) f(v, v)} g(v)^{\frac{2d}{\gamma^2}} dv\right) \frac{\frac{2}{\gamma}(Q-\gamma)}{\frac{2}{\gamma}(Q-\gamma) + 1} \overline{C}_{\gamma, d}.
\end{align*}

\noindent The tail asymptotics of $M_{\gamma, g}(A)$ then follows immediately from \Cref{cor:tau}.

\end{proof}

\appendix
\section{Reflection coefficient of GMC} \label{app:prob_rep}
In this appendix we explain why $\overline{C}_{\gamma, d}$ may be seen as a natural $d$-dimensional analogue of the Liouville reflection coefficients evaluated at $\gamma$. To commence with, we define $\overline{C}_{\gamma, d}(\alpha)$, which we call the reflection coefficient of GMC,  for each $\alpha \in (\frac{\gamma}{2}, Q)$ as follows.
\begin{prop}
Let $\overline{M}_{\gamma, \alpha} (0, r) = \int_{|x| \le r} |x|^{-\gamma \alpha} \overline{M}_{\gamma}(dx)$ for $\alpha \in (\frac{\gamma}{2}, Q)$. Then there exists some constant $\overline{C}_{\gamma, d}(\alpha) > 0$ independent of $r \in (0, r_d)$ such that
\begin{align}
\overline{C}_{\gamma, d}(\alpha) 
\notag & = \lim_{t \to \infty} t^{\frac{2}{\gamma}(Q - \alpha)} \PP\left(\overline{M}_{\gamma, \alpha}(0, r) > t \right)\\
\label{eq:GMC_coeff}
& = \lim_{\lambda \to 0^+} \frac{1}{\frac{2}{\gamma}(Q - \alpha)} \frac{\EE\left[ \overline{M}_{\gamma, \alpha}(0, r)^{\frac{2}{\gamma}(Q-\alpha)} e^{-\lambda \overline{M}_{\gamma, \alpha}(0, r)}\right]}{-\log \lambda}.
\end{align}
\end{prop}

\begin{proof}
The first equality can be obtained by a straightforward adaption of the proof of \Cref{lem:ref_GMC}, and the second equality follows from \Cref{lem:fake_tau}.
\end{proof}

We now show that $\overline{C}_{\gamma, d}(\alpha)$ coincides with the Liouville reflection coefficients\footnote{We only focus on $d=2$; for $d=1$ a similar proof shows that $\overline{C}_{\gamma, 1}$ coincides with the boundary unit volume reflection coefficient, see \cite[Section 4.3]{RV2017}.}.
\begin{prop}
When $d=2$, the reflection coefficient $\overline{C}_{\gamma, 2}(\alpha)$ of GMC is equivalent to the unit volume Liouville reflection coefficient $\overline{R}(\alpha)$ defined in \cite{RV2017}.
\end{prop}

\begin{proof}
Using the notations in \cite{RV2017}, we can write 
\begin{align*}
\overline{M}_{\gamma, \alpha} (0, 1)
\overset{d}{=} e^{\gamma M} \int_{-L_{-M}}^\infty e^{\gamma \Ba_s^{\alpha}} Z_s ds
=: e^{\gamma M} \Ia(L_{-M})
\end{align*}

\noindent  where 
\begin{itemize}
\item $Z_s ds$ is the GMC associated with the lateral noise of GFF;
\item $(\Ba_s^{\alpha})_{s \in \RR}$ an independent two-sided Brownian motion with negative drift $\alpha - Q$ conditioned to stay non-positive;
\item $M$ is an independent $\mathrm{Exp}(2(Q-\alpha))$ random variable; and 
\item $L_{-M}$ is the last time $(\Ba_s^\alpha)_{s \ge 0}$ hits $-M$.
\end{itemize}

\noindent Applying \eqref{eq:GMC_coeff} and the decomposition above, we have
\begin{align*}
\overline{C}_{\gamma, 2}(\alpha)
= \lim_{\lambda \to 0^+} \frac{1}{\frac{2}{\gamma}(Q - \alpha)} \EE\left[ \Ia(L_{-M})^{\frac{2}{\gamma}(Q- \alpha)} \left(\frac{ (e^{\gamma M})^{\frac{2}{\gamma}(Q-\alpha)} e^{-\lambda e^{\gamma M} \Ia(L_{-M})}}{-\log \lambda}\right)\right].
\end{align*}

\noindent When $\lambda \to 0^+$, the above expectation is dominated by the event that the exponential variable $M$ is large, in which case $L_{-M}$ is very large and $\Ia(L_{-M})$ behaves like $\Ia(\infty)$ which does not depend on $M$. To make this rigorous we aim to prove matching upper/lower bounds. Since $\PP(e^{\gamma M} > t) = t^{-\frac{2}{\gamma}(Q - \alpha)}$ for $t \ge 1$, a straightforward computation shows that
\begin{align*}
\EE \left[ \left(e^{\gamma M}\right)^{\frac{2}{\gamma}(Q-\alpha)} e^{-\lambda e^{\gamma M}}\right]
= -\frac{2}{\gamma}(Q-\alpha) e^{-\lambda} \log \lambda + O(1)
\end{align*}

\noindent where the error $O(1)$ is bounded independently of $\lambda > 0$. Using the fact that $\Ia(\infty)$ has moments of all orders smaller than $\frac{4}{\gamma^2}$ (\cite[Lemma 2.8]{KRV2017}), we deduce that
\begin{align*}
\overline{C}_{\gamma, 2}(\alpha)
& \le \lim_{\lambda \to 0^+} \frac{1}{\frac{2}{\gamma}(Q - \alpha)} \EE\left[ \Ia(\infty)^{\frac{2}{\gamma}(Q- \alpha)} \EE\left[\left(\frac{ (e^{\gamma M})^{\frac{2}{\gamma}(Q-\alpha)} e^{-\lambda e^{\gamma M} \Ia(0)}}{-\log \lambda}\right)\bigg| \Ia(0)\right] \right]\\
& = \EE \left[ \Ia(\infty)^{\frac{2}{\gamma}(Q-\alpha)}\right]
\end{align*}

\noindent which is the desired upper bound. Now fix any $T>0$, we have
\begin{align*}
\overline{C}_{\gamma, 2}(\alpha)
& \ge \lim_{\lambda \to 0^+} \frac{1}{\frac{2}{\gamma}(Q - \alpha)} \EE\left[ \Ia(L_{-T})^{\frac{2}{\gamma}(Q- \alpha)} \EE\left[\left(\frac{ (e^{\gamma M})^{\frac{2}{\gamma}(Q-\alpha)} e^{-\lambda e^{\gamma M} \Ia(\infty)}}{-\log \lambda}\right)\bigg| \Ia(\infty)\right] \right]\\
& \qquad - \lim_{\lambda \to 0^+} \frac{1}{\frac{2}{\gamma}(Q - \alpha)} \EE\left[ \Ia(\infty)^{\frac{2}{\gamma}(Q- \alpha)}\left(\frac{ (e^{\gamma M})^{\frac{2}{\gamma}(Q-\alpha)} e^{-\lambda e^{\gamma M} \Ia(\infty)}}{-\log \lambda}\right) 1_{\{M \le T\}}\right]\\
& = \EE \left[ \Ia(L_{-T})^{\frac{2}{\gamma}(Q-\alpha)}\right].
\end{align*}

\noindent Since $T$ is arbitrary, we may send $T \to \infty$ so that $L_{-T} \to \infty$ and obtain $\overline{C}_{\gamma, 2}(\alpha) \ge \EE \left[ \Ia(\infty)^{\frac{2}{\gamma}(Q-\alpha)}\right]$. This matches our upper bound and is precisely the probabilistic definition of the Liouville reflection coefficient $\overline{R}(\alpha)$ in \cite[equation (1.10)]{RV2017}.\\

\end{proof}



\end{document}